\documentclass[11pt]{article}

\usepackage{fancyhdr}
\usepackage{amsfonts}
\usepackage{amsmath}
\usepackage{amssymb}
\usepackage{amsthm}
\usepackage{tikz}
\usepackage{accents}
\usepackage{mathrsfs}
\usepackage{amscd}
\usepackage{amstext}
\usepackage{authblk}

\usepackage{graphics}
\usepackage{graphicx}

\newlength{\fixboxwidth}
\newcommand{\re}{\mathbb{R}}\newcommand{\N}{\mathbb{N}}
\newcommand{\zz}{\mathbb{Z}}\newcommand{\C}{\mathbb{C}}

\newcommand{\Z}{{\zz}^d}

\newcommand{\R}{{\re}^d}

\newcommand{\cs}{{\mathcal S}}

\newcommand{\cl}{{\mathcal L}}

\newcommand{\cf}{{\mathcal F}}
\newcommand{\cfi}{{\cf}^{-1}}

\newcommand{\supp}{{\rm supp \, }}

\newcommand{\dist}{{\rm dist \, }}

\newcommand{\gf}{\mathcal{F}}
\newcommand{\unif}{{\rm unif}}

\newcommand{\bproof}{\begin{proof}}
\newcommand{\eproof}{\end{proof}}

\newcommand{\fspq}{F^s_{p,q}\,}           %         Fspq

\newcommand{\be}{\begin{equation}}
\newcommand{\ee}{\end{equation}}
\newcommand{\beq}{\begin{eqnarray}}
\newcommand{\beqq}{\begin{eqnarray*}}
\newcommand{\eeq}{\end{eqnarray}}
\newcommand{\eeqq}{\end{eqnarray*}}

\numberwithin{equation}{section}

\newtheorem{theorem}{Theorem}[section]
\newtheorem{definition}[theorem]{Definition}
\newtheorem{corollary}[theorem]{Corollary}
\newtheorem{lemma}[theorem]{Lemma}
\newtheorem{proposition}[theorem]{Proposition}
\newtheorem{remark}[theorem]{Remark}

\begin{document}

\title{On a Problem of Jaak Peetre Concerning Pointwise Multipliers of Besov Spaces}

\author[a,b]{Van Kien Nguyen\thanks{E-mail: kien.nguyen@uni-jena.de,\ kiennv@utc.edu.vn}}
\author[a]{Winfried Sickel\thanks{E-mail: winfried.sickel@uni-jena.de}}
%email{kien.nguyen@uni-jena.de \quad \& \quad  winfried.sickel@uni-jena.de}
\affil[a]{Friedrich-Schiller-University Jena, Ernst-Abbe-Platz 2, 07737 Jena, Germany}
\affil[b]{University of Transport and Communications, Dong Da, Hanoi, Vietnam}

%%%%%%%%%%%%%%%%%%%%%%%%%%%%%%%%%%%%%%%%%%%%%%%%%%%%%%%%%%%%%%%%%%%%%%%%%%%%%%%%%%%%%
%%%%%%%%%%%%%%%%%%%%%%%%%%%%%%%%%%%%%%%%%%%%%%%%%%%%%%%%%%%%%%%%%%%%%%%%%%%%%%%%%%%%%

\date{\today}

\maketitle

\begin{abstract}  
We characterize the set of all pointwise multipliers of the
Besov spaces $B^s_{p,q}(\R)$ under the restrictions $0 < p,q \le \infty$ and  $s>d/p$.  
\end{abstract}

%{46E35}


%&&&&&&&&&&&&&&&&&&&&&&&&&&&&&&&&&&&&&&&&&&&&&&&&&&&&&&&&&&&&&&&&&&&&&&&&&&&&&&&&&&&&&&&&&&&&&&&&&&&&
%&&&&&&&&&&&&&&&&&&&&&&&&&&&&&&&&&&&&&&&&&&&&&&&&&&&&&&&&&&&&&&&&&&&&&&&&&&&&&&&&&&&&&&&&&&&&&&&&&&&&

\section{Introduction and main results}

%&&&&&&&&&&&&&&&&&&&&&&&&&&&&&&&&&&&&&&&&&&&&&&&&&&&&&&&&&&&&&&&&&&&&&&&&&&&&&&&&&&&&&&&&&&&&&&&&&&&&
%&&&&&&&&&&&&&&&&&&&&&&&&&&&&&&&&&&&&&&&&&&&&&&&&&&&&&&&&&&&&&&&&&&&&&&&&&&&&&&&&&&&&&&&&&&&&&&&&&&&&

In his famous book {\em New thoughts on Besov spaces}, page 151,  Jaak Peetre posed the problem 
to determine the set of all pointwise multipliers $M (B^s_{p,q}(\R))$ of the Besov space 
$B^s_{p,q}(\R)$ in case $s>d/p$. Now, more than 40 years later, we are able to present the 
complete solution to this problem.
To describe this  we need to introduce a few related classes of functions.
Here we are forced to distinguish between the cases $p \le q$,  $q<p< \infty$ and $p=\infty$.
First we deal with $p \le q$.

\begin{definition} \label{unif}
Let $\psi \in C_0^\infty (\R)$ be a nonnegative nontrivial function.
Let $0< p,q \le \infty$ and $s\in \re$.  Then 
$B^s_{p,q}(\R)_{\unif}$  denotes the collection of all tempered distributions 
$f \in \cs' (\R)$ such that 
\beqq
\| \, f\, | B^s_{p,q}(\R)_{\unif}\|:= \sup_{\lambda \in \R} \| \, f(\, \cdot \, )\, \psi (\, \cdot\, -\lambda) \, | B^s_{p,q}(\R)\|
<\infty \, .
\eeqq
\end{definition}

The spaces $ B^s_{p,q}(\R)_{\unif}$ are quasi-Banach spaces independent of the choice of $\psi$ (in the sense of equivalent quasi-norms).

\begin{theorem}\label{b-main}
 Let $0<p\leq q\leq \infty$, and $s>d/p$. Then 
 \be\label{ws-02}
 M(B^s_{p,q}(\R))=B^s_{p,q}(\R)_{\unif}
 \ee
 in the sense of equivalent quasi-norms.
\end{theorem}

In proving this theorem we will make use of the characterization of Besov spaces by differences in a way 
similar to Strichartz in his  paper \cite{Str-67}, see also the monographs of  Maz'ya and Shaposnikova \cite{MS1}, \cite{MS2}.

\begin{remark}
\rm 
There is a large number of references dealing with pointwise multipliers for Besov  or Lizorkin-Triebel spaces.
Here we selected only those which include characterizations of $M (B^s_{p,q}(\R))$.
\begin{itemize}
 \item Strichartz \cite{Str-67}  proved \eqref{ws-02} for $p=q=2$. In fact, he was dealing with the more 
general case of Bessel potential spaces
$H^s_p (\R)$, $s>d/p$, but $B^s_{2,2} (\R) = H^s_2 (\R)$ (in the sense of equivalent norms).
His main tool were consisting in characterizations of $H^s_p (\R)$ by differences. 
\item 
Peetre \cite{Pe}, page 151, proved \eqref{ws-02} for $1 \le p=q \le \infty $. He used a method nowadays called paramultiplication, which consists 
in a clever decomposition of the product in the Fourier image.
\item
Maz'ya and Shaposnikova, see \cite[Theorems 4.1.1, 5.3.1, 5.3.2, 5.4.1]{MS2}, proved \eqref{ws-02} for  $1 \le  p=q < \infty $.
Also these authors worked with  characterizations by differences.
\item 
Netrusov \cite{Net} proved characterizations of $M(B^s_{p,q}(\R))$ in cases $0<p=q\le 1$ and\\
$0 <p\le 1$, $q=\infty$ in Fourier analytic terms. This has been the first contribution to the case $p \neq q$.

\item Sickel \cite{Si} proved characterizations of $M(B^s_{p,p}(\R))$ in terms of capacities for all $p$, $0 <p< \infty$, and all  
$s> d/p$.
The used method here is again paramultiplication in connection with the Fourier analytic description of the spaces. 
\item
Smirnov and S. \cite{SS} have shown the identity \eqref{ws-02} in case $1\le  p,q\le \infty$
by using atomic characterizations of Besov spaces.
\item 
Triebel \cite[Proposition 2.22]{Tr06} proved a new characterization of $M(B^s_{p,p}(\R))$, where either $0 <p \le  \infty$ and $s> d/p$
or $0 < p \le 1$ and  $s=d/p$.  
\end{itemize}
\end{remark}

Now we turn to the slightly more complicated case $q <p$.
Here we have to introduce spaces with a  different type of localization.
%By $\mathring{\ell}_p $ we shall denote the unit ball in $\ell_p (\Z)$. 
%\be \label{c-mu}
% \Big(\sum_{\mu \in \Z} |C_\mu|^p \Big)^{1/p} \le 1\, .
%\ee

\begin{definition}\label{multdef}
Let $\psi \in C_0^\infty (\R)$ be a nonnegative  function satisfying 
\be\label{ws-03}
\sum_{\mu \in \Z} \psi (x-\mu) = 1 \qquad \mbox{for all}\quad x \in \R\, .
\ee 
Let $s > 0$,  $0 < p,q\le \infty$,  $m\in \N$ and $s < m \leq s + 1$. 
By using  $\psi_\mu (\, \cdot \, ):= \psi (\, \cdot\,  -\mu)$, $\mu \in \Z$, 
the space $M^s_{p,q}(\R)$ is the  collection of all $f \in L_1^{\ell oc}(\R)$  such that
\beq\label{ws-04}
\| \, f\, |M^s_{p,q}(\R)\|& := & 
\sup_{\|\, \{C_{\mu}\}_\mu\, |\ell_p (\Z)\|\le 1}\Bigg\{\Big\|\, f (\, \cdot\, )\, \Big(\sum_{\mu\in \Z} C_{\mu} \psi_{\mu}(\, \cdot \, )\Big) 
\, \Big|L_p(\R)\Big\|^q
\\
&&    + \quad 
\sum_{k=0}^{\infty}\bigg(2^{ksp} \sup_{|h| < 2^{-k}}  \sum_{\mu\in \Z} |C_{\mu}|^p    \big\|\Delta_h^m(\psi_{\mu}\, f )(\cdot) 
\big|  L_p(\R)\big\|^p \bigg)^{q/p} \Bigg\}^{1/q} <\infty \, .
\nonumber
\eeq
\end{definition}

We fix $\nu$.
By choosing $C_\mu := \delta_{\mu,\nu}$, $\mu \in \Z$, it is easily seen that 
the right-hand side in \eqref{ws-04} reduces to the quasi-norm of $\psi_{\nu}\, f $ in $B^s_{p,q}(\R)$, see Proposition \ref{diff}.
This observation  yields 
\beqq
M^s_{p,q}(\R)\hookrightarrow B^s_{p,q}(\R)_{\unif}\, .
\eeqq
It is not difficult to see that this embedding is proper in case $q <p$ (all details will be given below).
We have the following  nice characterization of $M(B^s_{p,q} (\R))$.

\begin{theorem}\label{nec1}
Let $0< q < p < \infty$ and $s>d/p$. Let $\psi_\mu$, $\mu \in \Z$, be as in Definition \ref{multdef}. 
Then  $f \in M(B^s_{p,q}(\R))$ if and only if $\sum_{\mu \in \Z} C_\mu \psi_\mu f$ belongs to $B^s_{p,q}(\R)$ 
for all $\{C_\mu\}_\mu \in \ell_p (\Z)$ and 
\beqq
\sup_{\| \{C_\mu\}_\mu \, |\ell_p(\Z) \|\le 1} \, \Big\|\sum_{\mu \in \Z} C_\mu \psi_\mu f\Big|B^s_{p,q}(\R)\Big\|  < \infty \, .
\eeqq
\end{theorem}

The second main result of our paper can now formulated as follows.

\begin{theorem}\label{final}
Let $ 0< q <p < \infty$ and $s>d/p$. 
Then we have
\beqq
M(B^s_{p,q}(\R))=M^s_{p,q}(\R)
\eeqq
in the sense of equivalent quasi-norms. 
\end{theorem}

We supplement our findings with the more easy case of $p=\infty$.

\begin{theorem}\label{final2}
Let $s>0$ and $0<q\leq \infty$. Then it holds $M(B^s_{\infty,q}(\R))=B^s_{\infty,q}(\R)$  in the sense of equivalent quasi-norms. 
\end{theorem}

\begin{remark}
\rm 
(i)  Theorem \ref{final} seems to be a novelty. We are not aware of any additional reference in this direction.
\\
(ii) The simple characterization of  $M(B^s_{\infty,q}B(\R))$ has been known, see \cite{SS}, but \cite{SS} was never published.
\\
(iii) Of course, Theorem \ref{final2} and Theorem \ref{b-main} overlap in case $p=q=\infty$. 
In this context it is of certain interest to notice that 
$M(B^s_{\infty,q}(\R)) = B^s_{\infty,q}(\R)_\unif$ if and only if $q=\infty$.
\end{remark}

Finally we collect results on the limiting situation $s=d/p>0$.

\begin{theorem}\label{final3}
(i) Let  $0<p = q\leq 1$ and $s=d/p$. Then 
 \beqq
 M(B^s_{p,p}(\R))=B^s_{p,p}(\R)_{\unif}
 \eeqq
holds in the sense of equivalent quasi-norms.
\\
(ii) Let $ 0< q <p \le 1$ and $s=d/p$. 
Then we have
\beqq
M(B^s_{p,q}(\R))=M^s_{p,q}(\R)
\eeqq
in the sense of equivalent quasi-norms.
\\
(iii)  Let $\psi_\mu$, $\mu \in \Z$, be as in Definition \ref{multdef}. 
Under the same restrictions as in (ii) \\
 $f \in M(B^s_{p,q}(\R))$ if and only if $\sum_{\mu \in \Z} C_\mu \psi_\mu f$ belongs to $B^s_{p,q}(\R)$ 
for all $\{C_\mu\}_\mu \in \ell_p (\Z)$ and 
\beqq
\sup_{\| \{C_\mu\}_\mu \, |\ell_p(\Z) \|\le 1} \, \Big\|\sum_{\mu \in \Z} C_\mu \psi_\mu f\Big|B^s_{p,q}(\R)\Big\|  < \infty \, .
\eeqq
\end{theorem}

\begin{remark}\label{abcd}
 \rm
 (i) In Corollary \ref{abc} below we shall prove that in the Banach space case Theorem \ref{b-main} and 
 Theorem \ref{final3} cover all cases where we have the coincidence $M(B^s_{p,q}(\R))=B^s_{p,q}(\R)_{\unif}$.
 \\
 (ii) Let $\Omega $ be  an open, nontrivial and  bounded subset of $\R$. 
 Define $B^s_{p,q} (\Omega)$ as the collection of all distributions $g \in \mathcal{D}'(\Omega)$ such that there exists 
 some $f \in B^s_{p,q} (\R)$ satisfying $f(\varphi) = g (\varphi)$ for all $\varphi \in \mathcal{D}(\Omega)$.
 Equipped with the quotient norm
 \[
  \|\, g \, | B^s_{p,q} (\Omega)\|:=  \inf \Big\{ \|\, f \, | B^s_{p,q} (\R)\| ~: \quad f_{|_\Omega} =g \Big\} 
 \]
$B^s_{p,q} (\Omega)$ becomes a quasi-Banach space.
Intrinsic characterizations, e.g., in the spirit of Proposition \ref{diff}, are known in case of a Lipschitz boundary, we refer to Dispa \cite{Di} and Triebel \cite[4.1.4]{Tr06}.
Essentially as a consequence of Theorem \ref{algebra} below it follows
\[
M(B^s_{p,q}(\Omega)) = B^s_{p,q} (\Omega)
\]
if $0 < p,q\le \infty$ and either $s>d/p$ or $s=d/p>0$ and $0 < q \le 1$.
Hence, the difficulties in determining $M(B^s_{p,q}(\R))$ are connected with the unboundedness 
of $\R$ and the difficult localization properties of Besov spaces, see Proposition \ref{s<p,q<u}.
\end{remark}

%&&&&&&&&&&&&&&&&&&&&&&&&&&&&&&&&&&&&&&&&&&&&&&&&&&&&&&&&&&&&&&&&&&&&&&&&&&&&&&&&&&&&&&&&&&&&&&&&&&
%&&&&&&&&&&&&&&&&&&&&&&&&&&&&&&&&&&&&&&&&&&&&&&&&&&&&&&&&&&&&&&&&&&&&&&&&&&&&&&&&&&&&&&&&&&&&&&&&&&

\subsection*{A short overview on further results in case $s\le d/p$}

%&&&&&&&&&&&&&&&&&&&&&&&&&&&&&&&&&&&&&&&&&&&&&&&&&&&&&&&&&&&&&&&&&&&&&&&&&&&&&&&&&&&&&&&&&&&&&&&&&&&
%&&&&&&&&&&&&&&&&&&&&&&&&&&&&&&&&&&&&&&&&&&&&&&&&&&&&&&&&&&&&&&&&&&&&&&&&&&&&&&&&&&&&&&&&&&&&&&&&&&&

As a service for the reader we finish this section with an overview about the knowledge on $M(B^s_{p,q}(\R))$ in case $s\le d/p$.

\begin{itemize}
 \item Let $1\le p < \infty$ and $0< s \le d/p$. Then $M(B^s_{p,p}(\R))$ has been characterized by 
Maz'ya and Shaposnikova, see \cite[Theorems 4.1.1, 5.3.1, 5.3.2, 5.4.1]{MS2}.
\item
Let  $0 <p =q \le 1$  and $d \, ( \frac 1p -1)< s \le d/p $. Then $M(B^s_{p,p}(\R))$ has been characterized by 
Netrusov \cite{Net}. Here we wish to mention that the description of $M(B^s_{p,p}(\R))$ given by Netrusov looks different compared to 
Theorem \ref{final3}.
\item
Let  $0 <p \le 1$  and $d\, (\frac 1p -1)< s \le d/p$. Then $M(B^s_{p,\infty}(\R))$ has been characterized by 
Netrusov \cite{Net}. 
\item
Let  $0 <p=q < \infty$  and $d\, \max (0, \frac 1p -1)< s \le d/p$. Then $M(B^s_{p,p}(\R))$ has been characterized by 
Sickel \cite{Si}.
\item 
Let $p=\infty$ and $s=0$. In \cite{KS} the spaces $M(B^s_{\infty,1}(\R))$ and $M(B^s_{\infty,\infty}(\R))$ 
have been characterized.
\item 
Triebel  \cite[Theorem 2.25]{Tr06} has found a new characterization of $M(B^s_{p,p}(\R))$, $0< p \le 1$ and $s> d(\frac 1p - 1)$,
in terms of the quantity
\[
\sup_{\mu \in \Z} \sup_{j \in \N_0} \, \| \psi_\mu (\, \cdot\, ) \, f(2^{-j}\cdot\,)\, |B^s_{p,p} (\R) \|
\] 
where $\psi_\mu$ is defined as in Definition \ref{multdef}, see also Schneider and Vybiral \cite{SV}.
\end{itemize}

The paper will be organized as follows.
In Section \ref{def} we collect all what we need about the function spaces under consideration.
This will be followed by a short section including basic properties of pointwise multipliers.
The next Section  is devoted to the proof of Theorem \ref{b-main} including some limiting cases 
with $s=d/p$.
In Section \ref{main4} we deal with the proof of Theorem \ref{final}.

%&&&&&&&&&&&&&&&&&&&&&&&&&&&&&&&&&&&&&&&&&&&&&&&&&&
%&&&&&&&&&&&&&&&&&&&&&&&&&&&&&&&&&&&&&&&&&&&&&&&&&&&

\subsection*{Notation}

%&&&&&&&&&&&&&&&&&&&&&&&&&&&&&&&&&&&&&&&&&&&&&&&&&&
%&&&&&&&&&&&&&&&&&&&&&&&&&&&&&&&&&&&&&&&&&&&&&&&&&&&

As usual, $\N$ denotes the natural numbers, $\N_0=\N\cup\{0\}$, 
$\zz$ denotes the integers, 
$\re$ the real numbers, 
and $\C$ the complex numbers.  For a real number $a$ we put $a_+ := \max(a,0)$. The letter $d\in \N, \ d>1,$ is always reserved for the underlying dimension in $\R$ and $\Z$.

If $X$ and $Y$ are two (quasi-)normed spaces, the (quasi-)norm
of an element $x$ in $X$ will be denoted by $\|x\,|\,X\|$. 
The symbol $X \hookrightarrow Y$ indicates that the
identity operator is continuous. For two sequences $a_n$ and $b_n$ we will write $a_n \lesssim b_n$ if there exists a
constant $c>0$ such that $a_n \leq c\,b_n$ for all $n$. We use $a_n \asymp b_n$ if $a_n \lesssim b_n$ and $b_n
\lesssim a_n$. 

Let $\cs(\R)$ be the Schwartz space of all complex-valued rapidly decreasing infinitely differentiable  functions on $\R$. 
The topological dual, the class of tempered distributions, is denoted by $\cs'(\R)$ (equipped with the weak topology).
The Fourier transform on $\cs(\R)$ is given by 
\[
\cf \varphi (\xi) = (2\pi)^{-d/2} \int_{\R} \, e^{-ix \xi}\, \varphi (x)\, dx \, , \qquad \xi \in \R\, .
\]
The inverse transformation is denoted by $\cfi $.
We use both notations also for the transformations defined on $\cs'(\R)$\,.

%&&&&&&&&&&&&&&&&&&&&&&&&&&&&&&&&&&&&&&&&&&&&&&&&&&&&&&&&&&&&&&&&&&&&&&&&&&&&&&&&&&&&&&&&&&&&&&
%&&&&&&&&&&&&&&&&&&&&&&&&&&&&&&&&&&&&&&&&&&&&&&&&&&&&&&&&&&&&&&&&&&&&&&&&&&&&&&&&&&&&&&&&&&&&&&

\section{Besov spaces}\label{def}

%&&&&&&&&&&&&&&&&&&&&&&&&&&&&&&&&&&&&&&&&&&&&&&&&&&&&&&&&&&&&&&&&&&&&&&&&&&&&&&&&&&&&&&&&&&&&&&
%&&&&&&&&&&&&&&&&&&&&&&&&&&&&&&&&&&&&&&&&&&&&&&&&&&&&&&&&&&&&&&&&&&&&&&&&&&&&&&&&&&&&&&&&&&&

General references for Besov spaces are, e.g., the  monographs of Nikol'skij \cite{Ni}, Peetre \cite{Pe} and 
Triebel \cite{Tr83}, \cite{Tr92}, \cite{Tr06}.
To introduce Besov spaces for the full range of parameters we make use of Fourier analysis, a way, 
originally introduced by Peetre \cite{Pe} and later propagated also by Triebel  \cite{Tr83}, \cite{Tr92}, \cite{Tr06}.

%&&&&&&&&&&&&&&&&&&&&&&&&&&&&&&&&&&&&&&&&&&&&&&&&&&&&&&&&&&&&&&&&&&&&&&&&&&&&&&&&&&&&&&&&&&&&&&
%&&&&&&&&&&&&&&&&&&&&&&&&&&&&&&&&&&&&&&&&&&&&&&&&&&&&&&&&&&&&&&&&&&&&&&&&&&&&&&&&&&&&&&&&&&&&&&

\subsection{Definition and basic properties}

%&&&&&&&&&&&&&&&&&&&&&&&&&&&&&&&&&&&&&&&&&&&&&&&&&&&&&&&&&&&&&&&&&&&&&&&&&&&&&&&&&&&&&&&&&&&&&&
%&&&&&&&&&&&&&&&&&&&&&&&&&&&&&&&&&&&&&&&&&&&&&&&&&&&&&&&&&&&&&&&&&&&&&&&&&&&&&&&&&&&&&&&&&&&

Let $\varphi_0 \in C_0^{\infty}({\R})$ be a non-negative function such that 
 $\varphi_0(x) = 1$ if $|x|\leq 1$ and $ \varphi_0 (x) =0$ if $|x|\geq 3/2$. 
For $k\in \N$ we define
\beqq
         \varphi_k(x) = \varphi_0(2^{-k}x)-\varphi_0(2^{-k+1}x) ,\qquad\ x \in \R\, . 
\eeqq
Because of 
\[
 \sum_{k=0}^\infty \varphi_k(x) = 1\, , \qquad x\in \R\, , 
\]
and
\[
 \supp \varphi_k \subset \big\{x\in \R: \: 2^{k-1}\le |x|\le 3 \cdot 2^{k-1}\big\}\, , \qquad k \in \N\, ,
\]
we  call the system $(\varphi_k)_{k\in \N_0 }$ a smooth  dyadic decomposition of unity on $\R$. 
Clearly, 
by the Paley-Wiener-Schwarz theorem, 
\be\label{ws-09}
f_k (x):= \cfi[\varphi_{k}\, \cf f](x)\, , \qquad x \in \R\, , \quad k \in \N_0\, .
\ee
is a smooth function for all $f\in \cs'(\R)$.

\begin{definition}\label{def-do}
Let $(\varphi_k)_{k\in \N_0 }$ be the above system. Let $0<  p,q\leq \infty$  and $s \in \re$.  Then  $ B^{s}_{p,q}(\re^d)$ is the
         collection of all tempered distributions $f \in \mathcal{S}'(\R)$
         such that
\beqq
          \|\, f \, |B^s_{p,q}(\R)\|_{\varphi_0} :=
         \bigg(\sum\limits_{k=0}^{\infty} 2^{k s q}\, \big\|\, \cfi[\varphi_{k}  \cf f](\, \cdot \, )\, 
         \big|L_p(\re^d)\big\|^q\bigg)^{1/q} <\infty
\eeqq   
(with usual modification if $q= \infty$).
\end{definition}

Clearly, Besov spaces are quasi-Banach spaces independent of the generator $\varphi_0 $ (in the sense of equivalent quasi-norms).
Therefore we will write $\|\, f \, |B^s_{p,q}(\R)\|$ instead of $\|\, f \, |B^s_{p,q}(\R)\|_{\varphi_0}$.
\\
For us, embeddings into $L_1^{\ell oc} (\R)$ and $L_\infty (\R)$ will be of some interest.

\begin{lemma}\label{emb1}
Let $0 < p,q\le \infty$ and $s \in \re$. 
\\
{\rm (i)} Let $s>d\, \max (0, \frac 1p -1)$. Then 
$B^s_{p,q} (\R)$ is continuously embedded into $L_{\max(1,p)} (\R)$.
\\
{\rm (ii)} The Besov space $B^s_{p,q} (\R)$ is continuously embedded into  $L_\infty (\R)$ if and only if 
either $s> d/p$ or $ s=d/p$ and $0< q \le 1$.
\end{lemma}

\begin{remark}\label{emb2}
 \rm
 (i)  As it is well-known,  $B^s_{p,q} (\R) \hookrightarrow L_\infty (\R)$ if and only if 
 $B^s_{p,q} (\R) \hookrightarrow C (\R)$. Here $C(\R)$ denotes the Banach space of all uniformly continuous functions on $\R$ equipped with the supremum norm.
 \\
(ii) The embeddings described in Lemma \ref{emb1} have a certain history.
We would like to mention at least Grisvard, Peetre, Golovkin, Stein, Zygmund, Besov, Il'yin and  Brudnij.
For the Banach space case we refer to the supplement written by Lizorkin in the russian edition \cite[Suppl.~1.7]{TL} 
of Triebel's monograph \cite{Tr83} for more details. 
The general quasi-Banach case has been considered in \cite{sitr}.
\end{remark}

%&&&&&&&&&&&&&&&&&&&&&&&&&&&&&&&&&&&&&&&&&&&&&&&&&&&&&&&&&&&&&&&&&&&&&&&&&&&&&&&&&&&&&&&&&&&&&&
%&&&&&&&&&&&&&&&&&&&&&&&&&&&&&&&&&&&&&&&&&&&&&&&&&&&&&&&&&&&&&&&&&&&&&&&&&&&&&&&&&&&&&&&&&&&&&&

\subsection{Tools - the characterization by differences and some inequalities}

%&&&&&&&&&&&&&&&&&&&&&&&&&&&&&&&&&&&&&&&&&&&&&&&&&&&&&&&&&&&&&&&&&&&&&&&&&&&&&&&&&&&&&&&&&&&&&&
%&&&&&&&&&&&&&&&&&&&&&&&&&&&&&&&&&&&&&&&&&&&&&&&&&&&&&&&&&&&&&&&&&&&&&&&&&&&&&&&&&&&&&&&&&&&

For us the characterization of Besov spaces by differences will be more important.
In case of  a multivariate function $f:\R\to \C$,  $m  \in \N$, $x, h  \in \R$,   we put
 \[
 \Delta_{h}^{m} f(x):= \sum_{\ell =0}^{m} (-1)^{m -\ell} \, \binom{m}{\ell} \, f(x + \ell h )\, .
 \]
The related modulus of smoothness is defined as 
\[
 \omega_m (f,t)_p := \sup_{|h|<t} \|\, \Delta_{h}^{m} f\, |L_p (\R)\|\, , \qquad t>0\, .
\]

\begin{proposition}\label{diff}
Let $0<  p,q\leq \infty$, $s> d\, \max(0, \frac 1p-1)$ and $s <m$ for some natural number $m$. Then the Besov space 
$B^s_{p,q}(\R)$ is a collection of all $f\in L_p(\R)$ such that 
\beqq
\|f|B^s_{p,q}(\R)\|_m := \|\, f\, |L_p(\R)\| + \bigg( \sum_{k=0}^{\infty} \, \big(2^{ks}\,  \omega_m (f, 2^{-k})\big)^q\bigg)^{1/q} < \infty. 
\eeqq
Furthermore, $\|\, \cdot \, |B^s_{p,q}(\R)\|_m $ and $\|\, \cdot\, |B^s_{p,q}(\R)\|$ are equivalent on 
$L_{\max(1,p)} (\R)$ for any admissible $m$.
\end{proposition}

\begin{remark}
\rm
The restriction   $s> d\, \max(0, \frac 1p-1)$ is natural in such a context. Since
$B^s_{p,q}(\R)$ contains singular distributions if $s < d\, \max(0, \frac 1p-1)$
a characterization as in Proposition \ref{diff} becomes impossible.
The version stated in Proposition \ref{diff} is a direct consequence of Theorem 2.5.12 in \cite{Tr83}
using the monotonicity of $\omega_m (f,t)_p $ with respect to $t$.
\end{remark}

The following construction of a maximal function is essentially  due to Peetre, but based on earlier work of Fefferman and Stein.
Let $a>0$ and  $b >0$ be fixed. For  $f \in L_1^{\ell oc} (\R)$ we define the Peetre maximal function $P_{b,a}f$ by
\beqq
  P_{b,a}f(x) := \sup\limits_{z\in \R} \frac{|f(x-z)|}{ 1+|b z |^{a}}\, , \qquad x \in \R\, .
\eeqq

\begin{proposition}\label{peetremax}
Let $0< p \leq\infty$ and define $\Omega :=\{x: |x|\leq b \}$ for some $b>0$. Let further $a>d/p$. 
Then there exists a positive constant $C$,  independent of $b$, such that
\beqq
\big\| P_{b,a}f \big|L_p(\R)\big\|\leq C\, \|f |L_p(\R)\|
\eeqq
holds for all $f \in L_{\max(1,p)}(\R)$ with $\supp (\gf f)\subset \Omega$.
\end{proposition}

For a proof we refer to \cite[Thm.~1.4.1]{Tr83}. 
A very useful  relation between Peetre maximal function and differences is given by the following 
lemma, see \cite{U1} and also \cite[page 102]{Tr83}.

 \begin{lemma}\label{1dim-1}
 Let $a>0$ and $m \in \N$.  
 Then there exists a constant $C$
 such that 
\beqq
     |\Delta^m_hf(x)| \leq  C\, \max (1,|bh|^a)\, \min(1,|bh|^m)\, P_{b,a}f(x)\,.
\eeqq
 holds for all $b > 0$, all $h,x\in \R$ and all $f\in \cs'(\R)$ satisfying $\supp(\gf f) \subset \{x: |x|\leq b \}$.  
 \end{lemma}

Later we shall  need also the following modification.

\begin{lemma}\label{help}
Let $a>0$ and   
$\psi \in C_0^{k}(\R)$ for some  $k \in \N$. Then, if $m\in \N$, $m \le k$,  there exists a constant $C$ such that 
\beqq
|\Delta^m_h (\psi \, \cdot \, f)(x)| \leq C \, \max\{1,|bh|^a\}\min\{1,|bh|^m\}P_{b,a}f(x)
\eeqq
holds for all $x,h\in \R$ and  all  $f\in \cs'(\R)$ such that $\supp (\gf f) \subset
\{x: |x|\leq b \}$, $b >0$. Here $C$ can be chosen independent of $b$.
\end{lemma}

\begin{proof}
The proof in one dimension can be found in \cite{NUU}. The general case follows by the same type of argument.  
\end{proof}

%&&&&&&&&&&&&&&&&&&&&&&&&&&&&&&&&&&&&&&&&&&&&&&&&&&&&&&&&&&&&&&&&&&&&&&&&&&&&&&&&&&&
%&&&&&&&&&&&&&&&&&&&&&&&&&&&&&&&&&&&&&&&&&&&&&&&&&&&&&&&&&&&&&&&&&&&&&&&&&&&&&&&&&&&

\section{Pointwise multipliers}\label{main}

%&&&&&&&&&&&&&&&&&&&&&&&&&&&&&&&&&&&&&&&&&&&&&&&&&&&&&&&&&&&&&&&&&&&&&&&&&&&&&&&&&&&
%&&&&&&&&&&&&&&&&&&&&&&&&&&&&&&&&&&&&&&&&&&&&&&&&&&&&&&&&&&&&&&&&&&&&&&&&&&&&&&&&&&&

%&&&&&&&&&&&&&&&&&&&&&&&&&&&&&&&&&&&&&&&&&&&&&&&&&&&&&&&&&&&&&&&&&&&&&&&&&&&&&&&&&&&
%&&&&&&&&&&&&&&&&&&&&&&&&&&&&&&&&&&&&&&&&&&&&&&&&&&&&&&&&&&&&&&&&&&&&&&&&&&&&&&&&&&&

\subsection{Some generalities on pointwise multipliers}\label{main1}

%&&&&&&&&&&&&&&&&&&&&&&&&&&&&&&&&&&&&&&&&&&&&&&&&&&&&&&&&&&&&&&&&&&&&&&&&&&&&&&&&&&&
%&&&&&&&&&&&&&&&&&&&&&&&&&&&&&&&&&&&&&&&&&&&&&&&&&&&&&&&&&&&&&&&&&&&&&&&&&&&&&&&&&&&

For a quasi-Banach space $X$ of functions  we shall call a function $f$ a pointwise multiplier
if $f \, \cdot \, g \in X$ holds for all $g \in X$
(this is includes, of course, that the operation $g \mapsto f \, \cdot \, g$ must be well defined for all $g\in X$).
If $X \hookrightarrow L_p (\Omega)$ for some $p$ (here $\Omega$ is a domain in $\R$),  
as a consequence of the Closed Graph Theorem, we obtain that the liner operator 
$T_f : ~ g \mapsto f \, \cdot \, g$, associated to such a pointwise multiplier, must be continuous in $X$,
see \cite[p.~33]{MS2}. As usual we put
\[
\|\, T_f\, |\cl (X) \|:= \sup_{\|g|X\|\le 1} \, \|f \, \cdot \, g\, |X \|\, .
\]
The collection of all pointwise multipliers for a given space $X$ will be denoted by $M(X)$ and equipped with the quasi-norm
\[
 \| \, f \, |M(X)\|:= \|\, T_f\, |\cl (X) \|\, .
\]
Of course, beside the natural interpretation of $f\, \cdot \, g$ as a pointwise product of functions,  one can define 
$f\, \cdot \, g$ on certain subsets of $\cs'(\R) \times \cs'(\R)$, in particular on $\cs(\R) \times \cs'(\R)$.
Those extensions of the product will not play an important role within this paper. For that reason we will skip details and refer to 
Johnsen \cite{Jo1} and the monograph \cite[4.2]{RS}.

Besov spaces are translation invariant, i.e., $\|\, g(\, \cdot \, - \mu)\, |B^s_{p,q}(\R)\|=\|\, g\, |B^s_{p,q}(\R)\|$
for all $\mu \in \R$. This implies that the associated multiplier space $M(B^s_{p,q}(\R))$ is translation invariant as well.
Let $f \in M(B^s_{p,q}(\R))$.
Let $\psi \in C_0^\infty (\R)$ be a nonnegative nontrivial function.
Since $C_0^\infty (\R) \subset B^s_{p,q}(\R)$ we conclude that $\psi (\, \cdot \, -\mu) \, \cdot \, f(\, \cdot \, ) \in B^s_{p,q}(\R)$
for all $\mu \in \R$. 
By assumption and translation invariance of $B^s_{p,q}(\R)$ we know that 
\[
 \|\psi (\, \cdot \, -\mu) \, \cdot \, f(\, \cdot \, )\, |B^s_{p,q}(\R)\|  \le 
\|\, f\, |M(B^s_{p,q}(\R))\|
\, \|\psi  \, |B^s_{p,q}(\R)\|\, , 
\]
which proves $M(B^s_{p,q}(\R)) \hookrightarrow B^s_{p,q}(\R)_\unif$.

\begin{lemma}\label{mult}
Let $0 <p,q\le \infty $ and $s\in \re$. Then it follows 
 $M(B^s_{p,q}(\R)) \hookrightarrow B^s_{p,q}(\R)_\unif$.
\end{lemma}

%&&&&&&&&&&&&&&&&&&&&&&&&&&&&&&&&&&&&&&&&&&&&&&&&&&&&&&&&&&&&&&&&&&&&&&&&&&&&&&&&&&&
%&&&&&&&&&&&&&&&&&&&&&&&&&&&&&&&&&&&&&&&&&&&&&&&&&&&&&&&&&&&&&&&&&&&&&&&&&&&&&&&&&&&

\subsection{On the algebra property of Besov spaces}\label{main2}

%&&&&&&&&&&&&&&&&&&&&&&&&&&&&&&&&&&&&&&&&&&&&&&&&&&&&&&&&&&&&&&&&&&&&&&&&&&&&&&&&&&&
%&&&&&&&&&&&&&&&&&&&&&&&&&&&&&&&&&&&&&&&&&&&&&&&&&&&&&&&&&&&&&&&&&&&&&&&&&&&&&&&&&&&

We shall call a quasi-Banach space of functions $X$ an algebra with respect to pointwise multiplication 
(for short a multiplication algebra) if $f\, \cdot \, g \in X$ for all $f,g\in X$ and there exists a constant $c$ such that 
\[
\|f \, \cdot \, g\, |X \| \le c \, \|\, f \,  |X \|\, \| \, g\, |X \|
\]
holds for all $f,g \in X$.

\begin{theorem}\label{algebra}
Let $0<p,q\leq \infty$ and $s\in \re$. Then  $B^s_{p,q}(\R)$ is a multiplication algebra if and only if one of the 
following conditions is satisfied 
\begin{itemize}
\item $s>d/p$;
\item $0<p<\infty$, $s=d/p$ and $q\leq 1$.
\end{itemize} 
\end{theorem}

\begin{remark}\rm
 The {\em if}-part with $p\ge 1$ has been proved for the first time  1970 in Peetre \cite{Pe1}. 
But he had called this assertion  well-known in \cite{Pe1}. The extension to the quasi-Banach case
has been obtained by Triebel \cite{Tr1}, \cite[2.6.2]{Tr78}. There also necessity of the above conditions 
is shown. 
However, Triebel had overlooked that $B^0_{\infty, q} (\R)$, $0< q \le 1$, is not an algebra. 
For this correction we refer to \cite[4.6.4,~4.8.3]{RS}.
\end{remark}

%&&&&&&&&&&&&&&&&&&&&&&&&&&&&&&&&&&&&&&&&&&&&&&&&&&&&&&&&&&&&&&&&&&&&&&&&&&&&&&&&&&&
%&&&&&&&&&&&&&&&&&&&&&&&&&&&&&&&&&&&&&&&&&&&&&&&&&&&&&&&&&&&&&&&&&&&&&&&&&&&&&&&&&&&

\subsection{Localized Besov spaces}\label{lokal}

%&&&&&&&&&&&&&&&&&&&&&&&&&&&&&&&&&&&&&&&&&&&&&&&&&&&&&&&&&&&&&&&&&&&&&&&&&&&&&&&&&&&
%&&&&&&&&&&&&&&&&&&&&&&&&&&&&&&&&&&&&&&&&&&&&&&&&&&&&&&&&&&&&&&&&&&&&&&&&&&&&&&&&&&&

This subsection has preparatory character.\\
Let $\psi$ be a non-negative $C_0^{\infty}(\R)$ function  such that \eqref{ws-03} is satisfied. 
As above we put $\psi_{\mu}(x):=\psi(x-\mu)$, $\mu\in \Z,\ x\in \R$.

\begin{definition}\label{def-unif}  
Let $0<p,q\leq \infty$ and $s \in \re$. 
\\
{\rm (i)} $B^{s,\ell oc}_{p,q}(\R)$ denotes the collection of all $g \in \cs'(\R)$ such that 
$ \varphi   \cdot  g \, \in B^s_{p,q}(\R)$ for all $\varphi \in C_0^\infty (\R)$.
\\
{\rm (ii)} 
Let $0 < u \leq \infty$. Then $B^{s}_{p,q,v}(\R)$ is the collection of all $ g\in B^{s,\ell oc}_{p,q}(\R)$ such that
\beqq
\| \, g\, |B^s_{p,q,v}(\R)\| := \bigg(\sum_{\mu\in \Z}\| \psi_{\mu}f|B^s_{p,q}(\R)\|^v \bigg)^{1/v} <\infty 
\eeqq
with the usual modification in case $u=\infty$.
\end{definition}

\begin{remark} 
\rm 
Obviously, the spaces $B^s_{p,q,\infty}(\R)$ and $B^s_{p,q}(\R)_{\unif}$ coincide.
In addition we wish to mention that  the classes $B^{s}_{p,q,v}$ are quasi-Banach spaces, 
independent of $\psi$ in the sense of equivalent quasi-norms. 
\end{remark}

\begin{proposition}\label{s<p,q<u}
Let $0<p,q,v\leq \infty$ and $s>d\max(0, \frac 1p-1)$.
\\
{\rm (i)} The embedding $B^s_{p,q,v}(\R)\hookrightarrow B^s_{p,q}(\R)$ holds if and only if $v \leq \min(p,q)$.
\\
{\rm (ii)} The embedding $B^s_{p,q}(\R)\hookrightarrow B^s_{p,q,v}(\R)$ holds if and only if $v\geq \max(p,q)$.
\end{proposition}

\begin{proof}
{\it Step 1.} Sufficiency.\\
{\it Substep 1.1.} Proof of (i). Let $v \leq \min(p,q)$ and let $m$ be a natural number with $m>s$.  
We suppose $g \in B^s_{p,q,v}(\R)$. 
In case $0< v < \infty$ this implies 
$\psi_\mu g \in L_{r} (\R)$ with $r:= \max(1,p)$, see Lemma \ref{emb1}, and at least formally
\beq\label{ws-07}
\Big\| \sum_{\mu \in \Z}    \psi_{\mu}g  \big | L_r(\R)\Big\|
&  \lesssim &
\bigg(\sum_{\mu \in \Z}   \big\| \psi_{\mu} g \big| L_r(\R)\big\|^r\bigg)^{1/r}\nonumber \\
& \lesssim &  \bigg(\sum_{\mu \in \Z}   \big\| \psi_{\mu} g \big| L_r(\R)\big\|^v\bigg)^{1/v}
   \lesssim   \| \, g \, | B^s_{p,q,v}(\R)\|\, .
\eeq 
However, convergence of $\sum_{\mu \in \Z}    \psi_{\mu}g $ in $L_r (\R)$ can be derived from \eqref{ws-07} as well
(because of $v<\infty$). This implies that  $\sum_{\mu \in \Z}    \psi_{\mu}g$ is a regular distribution. Hence, 
the following argument makes sense 
\beqq
&& \hspace{-0.7cm}
\Big\|\sum_{\mu\in \Z} \psi_{\mu}g \Big|B^s_{p,q}(\R) \Big\|^v_m \\
 &  \lesssim & 
\bigg(\sum_{\mu \in \Z}   \big\| \psi_{\mu}g \big | L_p(\R)\big\|^p\bigg)^{v/p} + \Bigg( \sum_{k=0}^{\infty} \bigg\{2^{sk } 
\sup_{|h|<2^{-k}}\Big\| \sum_{\mu\in \Z} |\Delta_{h}^m(\psi_{\mu}g)(\cdot)|\Big| L_p(\R )\Big\|\bigg\}^q\Bigg)^{v/q} 
\\
&\lesssim & \|\, g\, |B^s_{p,q,v}(\R)\|^v + \Bigg( \sum_{k=0}^{\infty} \bigg\{2^{skv} \sup_{|h|<2^{-k}}
\Big\| \Big( \sum_{\mu\in \Z} |\Delta_{h}^m(\psi_{\mu}g)(\cdot)|\Big)^v\Big| L_{p/v}(\R )\Big\|\bigg\}^{q/v}\Bigg)^{v/q} .
\eeqq
Since $ \Delta_h^m(\psi_{\mu}g)(x) \equiv 0$ if $|x-\mu|>c$ for some appropriate positive $c$ (independent of $\mu$) we get
\beqq
\bigg( \sum_{\mu\in \Z} |\Delta_{h}^m(\psi_{\mu}g)(x)|\bigg)^v \leq C  \sum_{\mu\in \Z} |\Delta_{h}^m(\psi_{\mu}g)(x)|^v
\eeqq
with a constant $C$ independent of $g$ and $x$. 
Inserting this into the previous inequality and applying triangle inequality  with respect to $L_{p/v} (\R)$ first, afterwards 
 with respect to $\ell_{q/v}$,
we obtain
\beqq
\Big\|\, \sum_{\mu\in \Z} \psi_{\mu}g \, \Big|B^s_{p,q}(\R)\Big\|^v  \lesssim   \sum_{\mu\in \Z}\big\| \psi_{\mu}g\, \big|B^s_{p,q}(\R) \big\|^v 
\, .
\eeqq
Because of $g = \sum_{\mu\in \Z} \psi_{\mu}g  $ in $L_r (\R)$, see \eqref{ws-07}, we 
have coincidence also almost everywhere and therefore also in  $B^s_{p,q}(\R)$ by using Proposition \ref{diff}.
Hence we conclude
\beqq
\big\|\,g \, \big|B^s_{p,q}(\R)\big\|  \lesssim   \| \, g\, |B^s_{p,q,v}(\R) \|
\, .
\eeqq
This implies sufficiency in (i) in case $0< v< \infty$.
Now we consider the case $ v=p=q = \infty$. For each $x\in \R$ we choose $\mu\in \Z$ such that $x\in \supp\psi_\mu$.   
Denote $$\Omega_{\mu}:= \big\{\nu\in \Z: \dist(\supp\psi_\mu,\supp\psi_\nu)\leq m\big\}$$
and observe that the cardinality of the sets $\Omega_\mu$ is uniformly bounded in $\mu$.
For $|h|<1$, \eqref{ws-03} and Proposition \ref{diff} in case $p=q=\infty$ yield 
\beqq
|h|^{-s}|\Delta_h^mg(x)|&\leq& \sum_{\nu\in\Omega_\mu}|h|^{-s}|\Delta_h^m (g\psi_\nu)(x)|\\
&\leq& \sum_{\nu\in \Omega_\mu}|h|^{-s}\| \Delta_h^m (g\psi_\nu)(\cdot)|C(\R)\| \leq C\, \|\, g\, |B^s_{\infty,\infty,\infty}(\R)\|\, , 
\eeqq
which implies
\beqq
\|\, g\, |B^s_{\infty,\infty}(\R)\|_m = \| \, g\, |C(\R)\| + \sup_{|h|<1}\sup_{x\in \R}|h|^{-s}| \Delta_h^mg(x)| \lesssim 
\|\, g\, |B^s_{\infty,\infty,\infty}(\R)\|\, .
\eeqq
{\it Substep 1.2.} Proof of (ii). Let $g \in B^s_{p,q} (\R)$ and assume that $v\geq \max(p,q)$. 
By Lemma \ref{emb1} we conclude that $g \in L_p (\R)$. It follows
\be\label{ws-10}
\bigg(\sum_{\mu \in \Z}   \big\| \psi_{\mu}g \big | L_p(\R)\big\|^v\bigg)^{1/v} \leq \bigg(\sum_{\mu \in \Z}   \big\| \psi_{\mu}g \big | L_p(\R)\big\|^p\bigg)^{1/p} 
\asymp \| \, g\, |L_p(\R)\| \lesssim \|\, g\, |B^s_{p,q}(\R)\|.
\ee
Next we  consider the term
\[
A(g):= \Bigg\{\sum_{\mu\in \Z}\bigg(\sum_{k=0}^{\infty}2^{skq}\sup_{|h| < 2^{-k}}\|\Delta_h^m(\psi_{\mu}g)(\cdot) |L_p(\R)\|^q \bigg)^{v/q}\Bigg\}^{1/v} \, .
\]
Since $v\ge q$ it follows
\[
 A(g)\le 
\Bigg\{\sum_{k=0}^{\infty}2^{skq}\bigg(\sum_{\mu\in \Z}\sup_{|h| < 2^{-k}}\|\Delta_h^m(\psi_{\mu}g)(\cdot) |L_p(\R)\|^v \bigg)^{q/v}\Bigg\}^{1/q}.
\]
For the next step of our estimate we shall 
use the convention that  $\varphi_\ell \equiv 0$ if $\ell < 0$ and define $g_{\ell} :=\gf^{-1}(\varphi_{\ell}\gf g)$, $\ell \in \zz$, see \eqref{ws-09}. 
It follows
\be \label{decom}
\psi_{\mu} g =\psi_{\mu}  \sum_{\ell \in \zz} \gf^{-1}(\varphi_{k+\ell}  \gf g)  = \sum_{\ell \in \zz}   \psi_{\mu} g_{k+\ell}  \, ,
\ee 
which is valid in $L_p (\R)$, see Lemma \ref{emb1}. Hence, using the monotonicity of the $\ell_r$-norms, we find
\beq \label{k-02}
 A(g)  &\leq & 
 \Bigg\{\sum_{k=0}^{\infty}2^{skq}\bigg(\sum_{\mu\in \Z}\sup_{|h| < 2^{-k}}
 \Big\|\sum_{\ell\in \zz} |\Delta_h^m(\psi_{\mu}g_{k+\ell})(\cdot) |\Big|L_p(\R)\Big\|^v \bigg)^{q/v}\Bigg\}^{1/q} 
 \nonumber
 \\
 & \leq & \Bigg\{\sum_{k=0}^{\infty}2^{skq}\bigg(\sum_{\mu\in \Z}\sup_{|h| < 2^{-k}}\Big\|\sum_{\ell\in \zz} |\Delta_h^m(\psi_{\mu}g_{k+\ell})(\cdot) 
 |\Big|L_p(\R)\Big\|^p \bigg)^{q/p}\Bigg\}^{1/q}\, .
\eeq 
Temporarily we assume  $p\leq  1$. By means of $|a+b|^p \le |a|^p + |b|^p$, $a,b \in \re$, this yields
\[
A(g) \leq \Bigg\{\sum_{k=0}^{\infty}2^{skq}\bigg(\sum_{\ell\in \zz}\sum_{\mu\in \Z} \sup_{|h| < 2^{-k}}\big\|\Delta_h^m(\psi_{\mu}g_{k+\ell})(\cdot) 
  |L_p(\R)\big\|^p \bigg)^{q/p}\Bigg\}^{1/q} \, .
\]
Recall from Substep 1.1, that  $\Delta_h^{m}(\psi_{\mu}g_{k+\ell})(x)=0 $ if $|x-\mu|>c$. In case $|x-\mu|<c$ we apply Lemma \eqref{help} with $\ell<0$ and obtain 
\beqq
\sup_{|h| < 2^{-k}} \, 
\big|\Delta_h^{m}(\psi_{\mu}g_{k+\ell})(x)\big| \leq C  2^{m\ell}  P_{2^{k+\ell},a} g_{k+\ell}(x)\, , \qquad x \in \R\, , 
\eeqq
which leads to 
\beqq 
 \sum_{\mu\in \Z}  \sup_{|h| < 2^{-k}} \big\| \Delta_h^{m}(\psi_{\mu}f_{k+\ell})(\cdot) \big| L_p(\R)  \big\|^p   & \lesssim &   
 2^{m\ell p} \,  \big\|  P_{2^{k+\ell},a} g_{k+\ell}(\cdot)   \big|L_p(\R)  \big\| 
 \\
 & \lesssim &     2^{m\ell p}  \, \big\|    g_{k+\ell}(\cdot)   \big|L_p(\R)  \big\|^p 
\eeqq 
as long as we choose $a>d/p$, see Proposition \ref{peetremax}. 
In case $\ell \ge 0$  we use the obvious elementary inequality
\beqq
\sum_{\mu\in \Z}\sup_{|h| < 2^{-k}} \big\| \Delta_h^{m}(\psi_{\mu}g_{k+\ell})(\cdot) \big| L_p(\R)  \big\|^p    
\lesssim  \big\|   g_{k+\ell}(\cdot)   \big|L_p(\R)  \big\|^p.
\eeqq
Altogether we have found
\be \label{Pfa}
 \sum_{\mu\in \Z}\sup_{|h| < 2^{-k}} \big\| \Delta_h^{m}(\psi_{\mu}g_{k+\ell})(\cdot) \big| L_p(\R)  \big\|^p 
  \lesssim   \min\big\{1 ,   2^{m\ell p}\big\}  \big\|   g_{k+\ell}(\cdot)   \big|L_p(\R)  \big\|^p
\ee  
for all $\ell\in \zz$. Inserting this into \eqref{k-02} we obtain
\beqq
 A(g)
  \lesssim \Bigg\{\sum_{k=0}^{\infty}2^{skq}\bigg(\sum_{\ell\in \zz}
  \min\big\{1 ,   2^{m\ell p}\big\}  \big\|   g_{k+\ell}(\cdot)   \big|L_p(\R)  \big\|^p \bigg)^{q/p}\Bigg\}^{1/q}.
\eeqq
Now, if $q/p\leq 1$, we conclude
\beq\label{ws-11}
 A(g) & \lesssim & 
 \bigg\{ \sum_{\ell\in \zz}\min\big\{1 ,   2^{m\ell q}\big\} \sum_{k=0}^{\infty}2^{skq} \big\|   g_{k+\ell}(\cdot)   \big|L_p(\R)  \big\|^q  \bigg\}^{1/q}
 \nonumber
 \\
  & \leq & \bigg\{ \sum_{\ell\in \zz}\min\big\{1 ,   2^{m\ell q}\big\}2^{-s\ell q} \sum_{k=0}^{\infty}2^{s(k+\ell)q} \big\|   g_{k+\ell}(\cdot)   \big|L_p(\R)  \big\|^q  \bigg\}^{1/q} 
\nonumber  
\\
  & \lesssim & \| \, g \, |B^s_{p,q}(\R)\|
\eeq
since $s < m$.
If $q/p >1$ we use the triangle inequality in $\ell_{q/p}$ and find
\beq\label{ws-12}
A(g)\lesssim \Bigg\{\sum_{\ell\in \zz} \min\big\{1 ,   2^{m\ell p}\big\} \bigg(\sum_{k=0}^{\infty}2^{skq} 
  \big\|   g_{k+\ell}(\cdot)   \big|L_p(\R)  \big\|^q \bigg)^{p/q}\Bigg\}^{1/p}   \lesssim \| \, g\, |B^s_{p,q}(\R)\| .
\eeq
Summarizing, the embedding $B^s_{p,q}(\R)\hookrightarrow B^s_{p,q,v}(\R)$ in case $0 < p \le 1$ follows from 
\eqref{ws-10}, \eqref{ws-11} and \eqref{ws-12}.\\
Now we turn to the case $p>1$. Inequality \eqref{k-02} yields
\beqq
A(g)  
 & \lesssim & \Bigg\{\sum_{k=0}^{\infty}2^{skq}\bigg(\sum_{\mu\in \Z}\Big(\sum_{\ell\in \zz}\sup_{|h| < 2^{-k}}
 \big\|  \Delta_h^m(\psi_{\mu}g_{k+\ell})(\cdot)  \big|L_p(\R)\big\|\Big)^p \bigg)^{q/p}\Bigg\}^{1/q} \\
&  \lesssim & \Bigg\{\sum_{k=0}^{\infty}2^{skq}\bigg(\sum_{\ell\in \zz}\Big(\sum_{\mu\in \Z}\sup_{|h| < 2^{-k}}
\big\|  \Delta_h^m(\psi_{\mu}g_{k+\ell})(\cdot)  \big|L_p(\R)\big\|^p\Big)^{1/p} \bigg)^{q }\Bigg\}^{1/q} \\
&  \lesssim & \Bigg\{\sum_{k=0}^{\infty}2^{skq}\bigg(\sum_{\ell\in \zz}\min\big\{1 ,   2^{m\ell }\big\}  
\big\|   g_{k+\ell}(\cdot)   \big|L_p(\R)  \big\| \bigg)^{q }\Bigg\}^{1/q} ,   
\eeqq
see \eqref{Pfa}. Next we divide into two cases: $q\geq 1$ and $q<1$. Then we can continue  as before and obtain
$B^s_{p,q}(\R)\hookrightarrow B^s_{p,q,v}(\R)$ also in case $1 \le p \le \infty$.\\
{\it Step 2.}  Necessity.\\
{\em Substep 2.1.} The $p$-dependence.
Without loss of generality we assume that $\psi\equiv 1$ on $[-\delta,\delta]^d$ for some 
$0<\delta<1$ and $\supp\psi \subset [-1,1]^d$. Let $\phi\in C_0^{\infty}(\R)$ be a nontrivial function  
such that $\supp \phi\subset [-\delta,\delta]^d $. 
Then we define a sequence $\{\mu_\ell\}_{\ell=0}^{n}$ with $\mu_\ell:=(4m\ell,0,\ldots,0)\in \Z$ and a sequence of functions
\[
g_n := \sum_{\ell=0}^{n} C_{\ell}\, \phi(x-\mu_\ell)\, , \qquad x \in \R\, , 
\]
where  the real numbers $C_\ell>0$, $\ell=1,\, \ldots\, ,n$, will be chosen later on.
Elementary calculations, based on the use of $\|\, \cdot \, |B^s_{p,q} (\R)\|_m$, yield
\beqq
\| \, g_n\, |B^s_{p,q,v}(\R)\| \asymp \Big(\sum_{\ell=0}^n C_{\ell}^v\Big)^{1/v} \qquad \text{and} \qquad \|\,  g_n\, |B^s_{p,q}(\R)\| \asymp \Big(\sum_{\ell=0}^n C_{\ell}^p\Big)^{1/p}.
\eeqq
with hidden constants independent of $n$ and $(C_\ell)_\ell$.
This implies the relations of  $u$ to $p$.
\\
{\em Substep 2.2.} The $q$-dependence. What concerns the $q$-dependence it is not longer convenient to work with differences.
We will switch to wavelets. Wavelet bases in Besov spaces are a
well-developed concept. We refer to the monographs of Meyer \cite{me},
Wojtasczyk \cite{woj} and Triebel \cite{Tr06,t08} for the general
$d$-dimensional case (for the one-dimensional case see also  the book of Kahane and Lemarie-Rieuseut
\cite{kl}). Let $\phi$ be a compactly supported sufficiently smooth scaling function  and
let
$\tilde\psi_1,\,\ldots,\,\tilde\psi_{2^d-1}$ be associated wavelets, all defined  on $\R$.
In addition we assume that $\supp \tilde{\psi}_1 \subset [-T,T]$ for some $T>0$. 
For $j\in \N_0$, $k\in \Z$, and $i=1,\ldots, 2^d-1$, we shall make use of the standard abbreviations in this context:
\[
\phi_{0,k}(x):=  \phi(x-k) \qquad \mathrm{and}\qquad 
\tilde\psi_{i,j,k}(x):= 2^{jd/2}\, \tilde\psi_i(2^jx-k),\qquad x\in\R.
\]
Then any $g \in B^s_{p,q}(\R)$  admits an unique representation

\begin{equation}\label{4.22}
g = \sum_{k\in\zz^d} a_k \, \phi_{0,k} + \sum_{i=1}^{2^d-1} \sum_{j=0}^\infty
\sum_{k\in\Z} a_{i,j,k}\, \tilde\psi_{i,j,k}
\end{equation}
in $\cs'(\R)$, where 
\[
a_k := \langle g,\,\phi_{0,k}\rangle \qquad \mbox{and} \qquad 
a_{i,j,k}:= \langle g,\,\tilde\psi_{i,j,k}\rangle \, .
\] 
Moreover,
\be\label{ws-13}
\|\, g \, |{B^s_{p,q}(\R)}\| \asymp
\bigg(\sum_{k\in\Z} |a_k|^p\bigg)^{1/p} + \Bigg[\sum_{i=1}^{2^d-1}\sum_{j=0}^\infty \, 2^{j(s+d/2)q}
\bigg(\sum_{k\in\Z} 2^{-jd}\, |a_{i,j,k}|^p\bigg)^{q/p}\Bigg]^{1/q},
\ee
in the sense of equivalent quasi-norms; see, e.\,g., \cite[Theorem 1.20]{t08}.
Now we define
\[
 g_{\alpha, \mu} := \sum_{j=1}^\infty \alpha_j \, \tilde\psi_{1,j,\mu_j}
\]
for some sequence $\alpha := (\alpha_j)_j$ of positive numbers   and some sequence $\mu:= (\mu_j)_j \subset \Z$ to be chosen later on.
It follows from \eqref{ws-13}
\be\label{ws-14}
\|\, g_{\alpha,\mu} \, |{B^s_{p,q}(\R)}\| \asymp
\bigg(\sum_{j=1}^\infty  \, 2^{j\big(s+d (\frac 12 - \frac 1p)\big)q} \, |\alpha_j|^q\bigg)^{1/q}
\ee
with hidden constants independent of $\alpha$ and $\mu$.
Without loss of generality we may assume that 
our function $\psi$ used in the definition of $B^s_{p,q,v} (\R)$ satisfies
$\psi\equiv 1$ on $[-\delta,\delta]^d$ for some 
$0<\delta<1$ and $\supp\psi \subset [-1,1]^d$.   
Next we choose a natural number $M$ such that 
\[
 \supp \tilde\psi_{1,j,0} \subset [-\delta,\delta]^d \qquad \mbox{for all}\quad j \ge M\, .
\]
Then it is easily checked that we get
\be\label{ws-15a}
\|\, g_{\alpha,\mu} \, |{B^s_{p,q,v}(\R)}\| \asymp
\bigg(\sum_{j=1}^\infty  \, 2^{j\big(s+d (\frac 12 - \frac 1p)\big)v} \, |\alpha_j|^v \bigg)^{1/v}
\ee
for all sequences $\alpha$ satisfying $\alpha_1 = \ldots = \alpha_M =0$ and 
$\mu$ chosen as $\mu_j:=(j\, 2^j,0,\ldots,0)$. 
Based on \eqref{ws-14} and \eqref{ws-15a},  the relations of $q$ to $v$ follow.
\end{proof}

\begin{remark}
 \rm 
Probably  Bourdaud \cite{Bou} was the first who had considered the classes $B^s_{p,q,v}(\R)$ with $p \neq q$ and $v \neq \infty$. 
He already investigated 
the embeddings in Proposition \ref{s<p,q<u} in case $v=p$.
\end{remark}

For convenience of the reader we state one obvious but important consequence of Proposition \ref{s<p,q<u}.

\begin{corollary}\label{klar}
Let $0 <p \le \infty$ and $s>d\max(0, \frac 1p-1)$. Then $B^s_{p,p}(\R)=B^s_{p,p,p}(\R)$ in the sense of equivalent quasi-norms. 
\end{corollary}

\begin{remark}
 \rm
 Corollary \ref{klar} is well-known, we refer to Peetre \cite[page 150]{Pe} ($p \ge 1$), 
Maz'ya and Shaposnikova \cite[Lem.~3.1.1.9]{MS1}, \cite[Prop.~4.2.6]{MS2} ($p \ge 1$),  
and Triebel \cite[2.4.7]{Tr92} (general case).
\end{remark}

%&&&&&&&&&&&&&&&&&&&&&&&&&&&&&&&&&&&&&&&&&&&&&&&&&&&&&&&&&&&&&&&&&&&&&&&&&&&&&&&&&&&
%&&&&&&&&&&&&&&&&&&&&&&&&&&&&&&&&&&&&&&&&&&&&&&&&&&&&&&&&&&&&&&&&&&&&&&&&&&&&&&&&&&&

\subsection{Proof of Theorem \ref{b-main}} \label{main3}

%&&&&&&&&&&&&&&&&&&&&&&&&&&&&&&&&&&&&&&&&&&&&&&&&&&&&&&&&&&&&&&&&&&&&&&&&&&&&&&&&&&&
%&&&&&&&&&&&&&&&&&&&&&&&&&&&&&&&&&&&&&&&&&&&&&&&&&&&&&&&&&&&&&&&&&&&&&&&&&&&&&&&&&&&

The heart of the matter consists in the following proposition.

\begin{proposition}\label{wichtig}
Let  $0<p\leq q\leq \infty$ and $s>d/p$.  Then there exists a constant $C>0$ such that
\be \label{k-01}
\|\, fg\, |B^s_{p,q}(\R)\| \leq C\, \| g\,|\, B^s_{p,q}(\R)\| \,  \|\, f\,|\, B^s_{p,q}(\R)_{\unif} \|
\ee 
holds for all $g\in B^s_{p,q}(\R)$ and all $f\in B^s_{p,q}(\R)_{\unif}$.
\end{proposition}

\begin{proof} 
Let $m-1\leq s < m$. For technical reasons we shall estimate $\|\, fg\, |B^s_{p,q}(\R)\|_{2m} $.
Let $\psi, \phi \in C_0^{\infty}(\R)$ are  chosen  such that the following holds:
$\psi$ is nontrivial and satisfies \eqref{ws-03} and $\phi \equiv 1$ on $\supp \psi $.  As explained  
in the proof of Proposition \ref{s<p,q<u} we have 
\beqq
\| \, fg\, |B^s_{p,q}(\R)\|_{2m} =\Big\| \sum_{\mu\in \Z} \phi_{\mu} g \psi_{\mu} f \, \Big|B^s_{p,q}(\R)\Big\|_{2m}\, .
\eeqq
Clearly,
\beq \label{ws-15}
\Big \| \sum_{\mu\in \Z} \phi_{\mu} g \psi_{\mu}f \Big| L_p(\R)\Big\| & \leq &
 \Big\| \sum_{\mu \in \Z} |\phi_{\mu}g | \cdot \| \psi_{\mu} f\, |L_\infty(\R)\|\,  \Big | L_p(\R)\Big\| 
\nonumber
\\
& \lesssim & \|\, g \, |L_p(\R)\| \cdot \, \sup_{\mu\in \Z}\,  \|\, \psi_{\mu}f\, |L_\infty(\R)\|
\nonumber
\\
& \lesssim & \| \, g\, |B^s_{p,q}(\R)\| \, \|\, f\, |B^s_{p,q}(\R)_{\unif}\|
\eeq
where we used in the last step the embedding $B^s_{p,q}(\R) \hookrightarrow L_\infty (\R)$, see Lemma \ref{emb1}.
Next we need some identities for differences.
Note that if $F,G:\ \R \to \C$ are two functions  and $n \in \N$ we have
\beqq
\Delta_h^{n}(F\, \cdot \, G)(x) = \sum_{j=0}^{n} \binom{n}{j} \, \Delta_h^{n-j}F(x+j h)\, \Delta_h^{j}G(x),\qquad x,h\in \R \, .
\eeqq 
This can be  proved by induction on $n$. 
Making use of this formula we obtain
\be\label{ws-15b}
|\Delta_h^{2m} (fg)(x)| \leq \sum_{u=0}^{2m} \binom{2m}{u} \sum_{\mu\in \Z}
\big| \Delta_h^{2m-u}(\phi_{\mu}g)(x+uh)\Delta_h^u(\psi_{\mu}f)(x)\big|,\ \ \ h,x\in \R.
\ee
Let $|h|\le 1$.
Since $ \Delta_h^u (\phi_{\mu}g)(x+uh) \equiv  \Delta_h^u (\psi_{\mu} f)(x) \equiv 0$ if $|x-\mu|>c$ for some appropriate positive $c$ 
(independent of $\mu$) this yields 
\beqq
S_u: & = & \Bigg\{\sum_{k=0}^{\infty} 2^{ksq} \sup_{|h|< 2^{-k}}  
\Big\|\sum_{\mu\in \Z}  \Delta_h^{2m-u}(\phi_{\mu}g)(\cdot+uh)\Delta_h^u(\psi_{\mu}f)(\cdot) \Big|L_p(\R)  
\Big\|^q \Bigg\}^{1/q}
\\ 
& \lesssim & \Bigg\{\sum_{k=0}^{\infty}\, 2^{ksq} \sup_{|h|< 2^{-k}}\, \bigg(  \sum_{\mu\in \Z}\big\| \Delta_h^{2m-u}(\phi_{\mu}g)(\cdot+uh)
\Delta_h^u(\psi_{\mu}f)(\cdot) \big|L_p(\R)  \big\|^p\bigg)^{q/p} \Bigg\}^{1/q}
\eeqq
for any $u$, $0 \le u \le 2m$. To  estimate $S_u$ we have to distinguish into different cases.\\
{\it Step 1.} The case $0 \le u < m$. Recall, $B^s_{p,q} (\R) \hookrightarrow L_\infty (\R)$ under the given restrictions, see Lemma \ref{emb1}. It follows
\beqq
\big\| \Delta_h^{2m-u}(\phi_{\mu}g)(\cdot +uh)&& \hspace{-0.9cm}\Delta_h^u(\psi_{\mu}f)(\cdot) \big|L_p(\R)  \big\| 
\\
 & \leq & \big\|   \Delta_h^{2m-u}(\phi_{\mu}g)(\cdot+uh)  \big|L_p(\R)  \big\|\cdot  \big\|\Delta_h^u(\psi_{\mu}f)(\cdot) \big| L_\infty (\R)\big\|
 \\
 & \lesssim & \big\|   \Delta_h^{2m-u}(\phi_{\mu}g)(\cdot+uh)  \big|L_p(\R)  \big\|\cdot  \big\| \psi_{\mu}f \, \big|\, L_\infty(\R)\big\|\\
 & \leq & \big\|   \Delta_h^{2m-u}(\phi_{\mu}g)(\cdot)  \big|L_p(\R)  \big\| \cdot \| \, f\, |\, B^s_{p,q }(\R)_{\unif}\|\, ,
\eeqq
where in the last step we performed a change of variable. Consequently, we obtain
\beqq
S_u \lesssim  \Bigg\{\sum_{k=0}^{\infty}\bigg(2^{ksp}\sup_{|h|< 2^{-k}} \sum_{\mu\in \Z} 
\big\| \Delta_h^{2m-u}(\phi_{\mu}g)(\cdot)  \big|L_p(\R)  
\big\|^p\bigg)^{q/p} \Bigg\}^{1/q} \cdot \| \, f \, |B^s_{p,q}(\R)_{\unif}\|\, .
\eeqq
Now we deal with the term $\{\, \ldots \, \}^{1/q}$ on the right-hand side. As in proof of Proposition \ref{s<p,q<u}, 
see formula \eqref{decom}, we use the decomposition
\beqq
\phi_{\mu} g = \sum_{\ell \in \zz} (\phi_{\mu} g_{k+\ell})\,.
\eeqq 
This yields
\beqq 
\Bigg\{\sum_{k=0}^{\infty} && \hspace{-0.6cm}
\bigg(2^{ksp}\sup_{|h|< 2^{-k}} \sum_{\mu\in \Z} \big\| \Delta_h^{2m-u}(\phi_{\mu}g)(\cdot)  \big|L_p(\R)  
\big\|^p\bigg)^{q/p} \Bigg\}^{1/q} 
\\
& \le & \Bigg\{\sum_{k=0}^{\infty}\bigg(2^{ksp}\sup_{|h|< 2^{-k}} \sum_{\mu\in \Z} 
\Big\| \sum_{\ell \in \zz} |\Delta_h^{2m-u}(\phi_{\mu}g_{k+\ell})(\cdot)|\,   \Big|L_p(\R)  
\Big\|^p\bigg)^{q/p} \Bigg\}^{1/q} \, .
\eeqq
Since $2m-u \ge m$ we can proceed as in Substep 1.2, proof of Proposition \ref{s<p,q<u}, starting at formula \eqref{k-02}.
As a result we find
\be\label{ws-17}
S_u\lesssim  \| \, g \, | B^s_{p,q}(\R)\| \,   \| f| B^s_{p,q}(\R)_{\unif}\|
\ee
for all $u$, $0 \le u \le m$.
\\
{\it Step 2.} The case $m < u \le 2 m$. We have
\beqq
\big\| \Delta_h^{2m-u}(\phi_{\mu}g)(\cdot+uh)  \Delta_h^u(\psi_{\mu}f)(\cdot) \big|  L_p(\R)\big\| 
& \leq & 
\big\| \Delta_h^{2m-u}(\phi_{\mu}g)\big|L_\infty(\R) \big\| \, \big\|\Delta_h^u(\psi_{\mu}f) \big|  L_p(\R)\big\|
 \\
 & \lesssim &   \big\| \phi_{\mu}g  \big|L_\infty(\R) \big\| \,  \big\|\Delta_h^u(\psi_{\mu}f) \big|  L_p(\R)\big\|\, .
\eeqq
Inserting this into $S_u$ we find
\be\label{ws-23}
S_u \leq \Bigg\{\sum_{k=0}^{\infty}\bigg(2^{ksp} \sup_{|h|< 2^{-k}}  
\sum_{\mu\in \Z}\big\| \phi_{\mu}g \big|L_\infty (\R) \big\|^p \,  \big\|\Delta_h^u(\psi_{\mu}f)(\cdot) \big|  L_p(\R)\big\|^p 
\bigg)^{q/p} \Bigg\}^{1/q}.
\ee
Using the triangle inequality in $L_{q/p}$ we arrive at
\beqq
S_u &  = & \Bigg\{\sum_{\mu\in \Z} \bigg( \big\| \phi_{\mu}g \big|L_\infty(\R) \big\|^q  \sum_{k=0}^{\infty} 2^{ksq}    
\sup_{|h|< 2^{-k}} \big\|\Delta_h^u(\psi_{\mu}f) \big|  L_p(\R)\big\|^q \bigg)^{p/q} \Bigg\}^{1/p} 
\\
&\leq &  \bigg\{\sum_{\mu\in \Z}   \big\| \phi_{\mu}g \big|L_\infty (\R) \big\|^p  \, 
 \big\|\,  \psi_{\mu}f\, |B^s_{p,q}(\R)\|^p \bigg\}^{1/p} 
\\
& \leq &\bigg\{\sum_{\mu\in \Z}   \big\| \phi_{\mu}g \big|L_\infty (\R) \big\|^p  \bigg\}^{1/p}  \,  
\big\|\, f\, |B^s_{p,q}(\R)_{\unif}\|.
\eeqq
Since $s>d/p$, there exists $\varepsilon>0$ such that $s-\varepsilon>d/p$. 
This implies $B^{s-\varepsilon}_{p,p}(\R)\hookrightarrow L_\infty(\R)$, see Lemma \ref{emb1}. Hence we obtain the estimate
\be\label{ws-19}
\bigg( \sum_{\mu\in \Z} \big\| \phi_{\mu}g \big|L_\infty (\R) \big\|^p    \bigg)^{1/p} 
\lesssim  \bigg( \sum_{\mu\in \Z} \big\| \phi_{\mu}g \big|B^{s-\varepsilon}_{p,p}(\R) \big\|^p    \bigg)^{1/p} 
\asymp \| \, g\, |B^{s-\varepsilon}_{p,p}(\R)\|
\ee
where we  used Corollary \ref{klar} with respect to $B^{s-\varepsilon}_{p,p}(\R)$. 
The elementary embedding $B^s_{p,q}(\R) \hookrightarrow B^{s-\varepsilon}_{p,p}(\R)$ implies 
\be\label{ws-18}
S_u\lesssim  \| \, g\,|\, B^s_{p,q}(\R)\| \,  \|\, f\,|\, B^s_{p,q}(\R)_{\unif} \|
\ee
also in this situation.
Summarizing, \eqref{ws-15}, \eqref{ws-15b}-\eqref{ws-18}, prove the claim \eqref{k-01}.
\end{proof}

\begin{remark}
 \rm
Let us mention that the inequality \eqref{ws-19} is not true in the limiting case $s=d/p$, $p < q\leq 1$.
From the Sobolev-type embedding $B^{d/p}_{p,q}(\R) \hookrightarrow B^{d/q}_{q,q}(\R)$, see \cite[2.7.1]{Tr83}, we derive
\beqq
\bigg( \sum_{\mu\in \Z} \big\| \phi_{\mu}g \big|L_\infty (\R) \big\|^q    \bigg)^{1/q} 
\lesssim  \bigg( \sum_{\mu\in \Z} \big\| \phi_{\mu}g \big|B^{d/q}_{q,q}(\R) \big\|^q    \bigg)^{1/q} 
\lesssim \| \, g\, |B^{d/p}_{p,q}(\R)\|\, .
\eeqq
The exponent $q$ is best possible, i.e., 
if 
\[
 \bigg( \sum_{\mu\in \Z} \big\| \phi_{\mu}g \big|L_\infty (\R) \big\|^v    \bigg)^{1/v} \lesssim \| \, g\, |B^{d/p}_{p,q}(\R)\|
\]
holds for all  $g \in B^{d/p}_{p,q}(\R)$, then $v \ge q$ follows.
\end{remark}

\begin{proposition}\label{limes}
Let  $0 <p \leq 1$ and $s=d/p$.  Then there exists a constant $C>0$ such that
\beqq
\|\, fg\, |B^{d/p}_{p,p}(\R)\| \leq C\, \| g\,|\, B^{d/p}_{p,p}(\R)\| \,  \|\, f\,|\, B^{d/p}_{p,p}(\R)_{\unif} \|
\eeqq 
holds for all $g\in B^s_{p,p}(\R)$ and all $f\in B^s_{p,p}(\R)_{\unif}$.
\end{proposition}

\begin{proof}
Let $\phi$ and $\psi$ be as in proof of Proposition \ref{wichtig}.
 Employing Corollary \ref{klar} and the algebra property of $B^{d/p}_{p,p}(\R)$, see Theorem \ref{algebra}, 
we get
\beqq
\| \, fg\,|\, B^s_{p,p}(\R)\| & \asymp & \bigg(\sum_{\mu \in \Z} 
\| \, (\psi_{\mu} f )(\phi_\mu g)\,|\, B^s_{p,p}(\R)\|^p \bigg)^{1/p}
\\
& \lesssim &   \bigg(\sum_{\mu \in \Z} \| \, \psi_\mu f \,|\, B^s_{p,p}(\R)\|^p
\| \, \phi_\mu g\,|\, B^s_{p,p}(\R)\|^p \bigg)^{1/p}
\\
& \lesssim &   \|\, f \, |B^s_{p,p}(\R)_\unif\|\, 
\bigg(\sum_{\mu \in \Z} \| \, \phi_\mu g\,|\, B^s_{p,p}(\R)\|^p \bigg)^{1/p}
\\
& \lesssim &   \|\, f \, |B^s_{p,p}(\R)_\unif\|\,\,   \| \,  g\,|\, B^s_{p,p}(\R)\|\, ,
\eeqq
which proves the claim.
\end{proof}

\begin{remark}
 \rm
 As already mentioned above this result can be found  in Netrusov \cite{Net} and Triebel \cite[Proposition 2.22]{Tr06}. 
\end{remark}

\noindent
{\bf Proof of Theorem \ref{b-main}.}
From Proposition \ref{wichtig} we derive
$   B^s_{p,q}\R)_\unif \hookrightarrow M(B^s_{p,q}(\R))$, whereas from 
Lemma \ref{mult}
$M(B^s_{p,q}(\R))  \hookrightarrow B^s_{p,q}(\R)_{\unif}$ follows.

%&&&&&&&&&&&&&&&&&&&&&&&&&&&&&&&&&&&&&&&&&&&&&&&&&&&&&&&&&&&&&&&&&&&&&&&&&&&&&&&&&&&
%&&&&&&&&&&&&&&&&&&&&&&&&&&&&&&&&&&&&&&&&&&&&&&&&&&&&&&&&&&&&&&&&&&&&&&&&&&&&&&&&&&&

\subsection{Proofs of Theorems \ref{final} and \ref{nec1}} \label{main4}

%&&&&&&&&&&&&&&&&&&&&&&&&&&&&&&&&&&&&&&&&&&&&&&&&&&&&&&&&&&&&&&&&&&&&&&&&&&&&&&&&&&&
%&&&&&&&&&&&&&&&&&&&&&&&&&&&&&&&&&&&&&&&&&&&&&&&&&&&&&&&&&&&&&&&&&&&&&&&&&&&&&&&&&&&

\begin{lemma}\label{unif-L}
 Let $0<p,q\leq \infty$ and $s >d(1/p-1)_+$. Then we have
\beqq
B^s_{p,q}(\R)\hookrightarrow M^s_{p,q}(\R)\hookrightarrow  B^s_{p,q}(\R)_{\unif}.
\eeqq 
\end{lemma}

\begin{proof}
The right-hand embedding has been explained just before Theorem \ref{final}. 
Also the left-hand embedding is easily seen. From the trivial inequality
\beqq
\| g|M^s_{p,q}(\R)\|\leq 
\|g|L_p(\R)\|+ \Bigg\{\sum_{k=0}^{\infty}\bigg(2^{ksp} 
\sup_{|h| < 2^{-k}}  \sum_{\mu\in \Z}     \big\|\Delta_h^m(\psi_{\mu}g)(\cdot) \big|  L_p(\R)\big\|^p \bigg)^{q/p} \Bigg\}^{1/q}
\eeqq
and the same argument as used in Substep 1.2 of the proof of Proposition \ref{s<p,q<u}, see formula \eqref{k-02},   
the left-hand embedding follows.
\end{proof}

\noindent
{\bf Proof of Theorem \ref{final}}. 
{\it Step 1.} Let $m$ be a natural number such that $m-1 < s < m$. We claim that 
\beqq
\| \, fg\, |B^s_{p,q}(\R)\|_{2m} \leq C\,  \|\, f\, |M^s_{p,q}(\R)\| \, \| \, g\, |B^s_{p,q}(\R)\| 
\eeqq
holds for all $f\in M^s_{p,q}(\R)$ and $g\in B^s_{p,q}(\R)$. As a first step of the proof we observe that 
in \eqref{ws-17} we did not use the condition $p\le q$. Hence we can apply \eqref{ws-17} for all terms $S_u$ with $0 \le u \le m$. 
Since $M^s_{p,q}(\R)\hookrightarrow  B^s_{p,q}(\R)_{\unif}$ this is sufficient for us.  It remains to deal  with the case $ m < u \le 2m$. 
Starting point for us is formula \eqref{ws-23}
\[
S_u \leq \Bigg\{\sum_{k=0}^{\infty}\bigg(2^{ksp} \sup_{|h| < 2^{-k}}  \sum_{\mu\in \Z}\big\| \phi_{\mu}g \big|L_\infty(\R) \big\|^p 
\cdot  \big\|\Delta_h^m(\psi_{\mu}f)(\cdot) \big|  L_p(\R)\big\|^p \bigg)^{q/p} \Bigg\}^{1/q}.
\] 
Now we choose 
\[
C_{\mu}:= c \, \frac{ \| \phi_{\mu}g  |L_\infty(\R) \| }{ \|\, g\,  |B^{s}_{p,q}(\R) \| }\, , \qquad \mu\in \Z,
\]
where $c$ will be fixed later on.
Of course, here we assume that $\|\, g \,  |B^{s}_{p,q}(\R) \|>0$,  otherwise there is nothing to prove. 
With a proper choice of $c$, the sequence $\{C_{\mu} \}_{\mu\in \Z}$ belongs to the set $\ell_p^+$ because of 
\[
\sum_{\mu \in \Z} C_\mu^p = c^p \, \sum_{\mu \in \Z} \frac{ \| \phi_{\mu}g  |L_\infty(\R) \|^p }{ \|\, g\,  |B^{s}_{p,q}(\R) \|^p}
\le 1\, , 
\]
see \eqref{ws-19}. Notice that $c$ can be chosen independent of $g$.
Hence, we conclude that
\beqq
S_u \lesssim  \big\| \, g\, \big|B^s_{p,q}(\R)\big\|\cdot \Bigg\{\sum_{k=0}^{\infty}\bigg(2^{ksp} \sup_{|h| < 2^{-k}}  
\sum_{\mu\in \Z} C_{\mu}^p \big\|\Delta_h^m(\psi_{\mu}f)(\cdot) \big|  L_p(\R)\big\|^p \bigg)^{q/p} \Bigg\}^{1/q}  
\eeqq
which proves $M^s_{p,q}(\R)\hookrightarrow M(B^s_{p,q}(\R))$.\\
{\it Step 2.} Let $f \in  M(B^s_{p,q}(\R))$  and $\{C_{\mu}\}_{\mu}\in \ell_p (\Z)$.
We subdivide $\Z$ into a finite number of disjoint sets $\Omega_\ell$, $\ell=1,\ldots,n$,  such that 
\[
\dist\big(\supp  \psi_{\mu_i}, \supp \psi_{\mu_j}\big)\geq 2m\, , 
\]
for all $\mu_i,\mu_j\in \Omega_\ell$, $\mu_i \not=\mu_j$, and all $\ell=1,\ldots,n$. 
With 
\[
 P_\mu := \{x \in \R: \: \dist (\supp \psi_\mu,x)\le m\}\, , \qquad \mu \in \Z\, , 
\]
we find
\[
  \sum_{\mu\in \Omega_\ell} |C_{\mu}|^p    \int_{P_\mu} |\Delta_h^m(\psi_{\mu}f)(x) |^p dx = \int_{\R} \Big|\Delta_h^m \Big(
\sum_{\mu\in \Omega_\ell} C_{\mu} \psi_{\mu} f \Big)(x) \Big|^p dx \,, \qquad |h|< 1.
\]
This implies
\beq \label{k-06}
\Bigg\{\sum_{k=0}^{\infty}&& \hspace{-0.7cm}\bigg(2^{ksp} 
\sup_{|h| < 2^{-k}}  \sum_{\mu\in \Z} |C_{\mu}|^p    \big\|\Delta_h^m(\psi_{\mu}f)(\cdot) \big|  L_p(\R)\big\|^p \bigg)^{q/p} \Bigg\}^{1/q} 
\nonumber
\\
&\lesssim & \sum_{\ell=1,\ldots,n} \Bigg\{\sum_{k=0}^{\infty}\bigg(2^{ksp} \sup_{|h| < 2^{-k}}  \sum_{\mu\in \Omega_\ell} 
|C_{\mu }|^p    \big\|\Delta_h^m(\psi_{\mu }f)(\cdot) \big|  L_p(\R)\big\|^p \bigg)^{q/p} \Bigg\}^{1/q}
\nonumber
\\
& \lesssim & \sum_{\ell=1,\ldots,n} \Big\|  \sum_{\mu \in\Omega_\ell} C_{\mu } \psi_{\mu }f \, \Big|B^s_{p,q}(\R)\Big\|.
\eeq 
Next, we make use of the definition of $M(B^s_{p,q}(\R))$ 
\be \label{k-05}
\Big\| f \sum_{\mu\in \Omega_\ell } C_{\mu } \psi_{\mu }\Big|B^s_{p,q}(\R)\Big\| \lesssim 
\big\| f\big| M(B^s_{p,q}(\R)) \big\|  \cdot \Big\|   \sum_{\mu \in \Omega_\ell} C_{\mu } \psi_{\mu }\Big|B^s_{p,q}(\R)\Big\| .
\ee 
A simple calculation leads to
\beqq
 && \hspace{-0.7cm} \Big\|   \sum_{\mu \in \Omega_\ell} C_{\mu }\psi_{\mu }\Big|B^s_{p,q}(\R)\Big\|_m  
\\
&\lesssim &  \Big\|   \sum_{\mu \in \Omega_\ell} C_{\mu } \psi_{\mu }\Big|L_{p}(\R)\Big\| +
\Bigg\{\sum_{k=0}^{\infty} 2^{ksq} \sup_{|h| < 2^{-k}}  \bigg(\sum_{\mu\in \Omega_\ell} |C_{\mu}|^p 
\big\|\Delta_h^m\psi_{\mu} (\cdot) \big|L_p(\R) \big\|^p \bigg)^{q/p} \Bigg\}^{1/q}
\\
&\lesssim & \bigg(   \sum_{\mu \in \Omega_\ell} |C_{\mu }|^p \|\,  \psi\, |L_{p}(\R)\|^p\bigg)^{1/p} +
\Bigg\{\sum_{k=0}^{\infty} 2^{ksq} \sup_{|h| < 2^{-k}}  \bigg(\sum_{\mu\in \Z} |C_{\mu}|^p 
\big\|\Delta_h^m\psi (\cdot) \big|L_p(\R) \big\|^p \bigg)^{q/p} \Bigg\}^{1/q} 
\\
& \lesssim & \|\, \{C_\mu\}_\mu\, |\ell_p (\Z)\|\, \big\| \psi\big|B^s_{p,q}(\R)\big\|.
\eeqq
Inserting this into \eqref{k-05} we find
\be\label{ws-31}
\sup_{\|\, \{C_\mu\}_\mu\, |\ell_p (\Z)\| \le 1} \, 
\Big\| f \sum_{\mu\in \Omega_\ell } C_{\mu } \psi_{\mu }\Big|B^s_{p,q}(\R)\Big\| \lesssim \big\|\,  f\, \big| M(B^s_{p,q}(\R)) \big\|   
\ee
for all $\ell=1,\ldots,n$. 
Consequently, we obtain $\big\| f\big| M^s_{p,q}(\R) \big\| \lesssim \big\| f\big| M(B^s_{p,q}(\R)) \big\| $. The proof is complete.
\hspace*{\fill} \rule{3mm}{3mm} 
\\

In Step 2 of the preceding proof we did not use the restrictions in $p,q$ and $s$. We only used the possibility to describe 
the quasi-norm of $B^s_{p,q} (\R)$ by differences as explained in Proposition \ref{diff}. This yields the following.

\begin{lemma}\label{nec}
Let $ 0< p,q \le \infty$ and $s> d (\frac 1p -1)_+$. Then we have the  continuous embedding 
\beqq
M(B^s_{p,q}(\R)) \hookrightarrow M^s_{p,q}(\R)\, .
\eeqq
\end{lemma}

\noindent
{\bf Proof of Theorem \ref{nec1}}.
{\em Step 1.}
Let $f\in M(B^s_{p,q}(\R))$ and $(C_\mu)_\mu \in \ell_p (\Z)$. 
Then  \eqref{ws-31} yields what we need. 
\\
{\em Step 2.} Let $f$ be a function such that
 $\sum_{\mu \in \Z} C_\mu \psi_\mu f$ belongs to $B^s_{p,q}(\R)$ 
for all $\{C_\mu\}_\mu \in \ell_p (\Z)$ and 
\beqq
\sup_{\| \{C_\mu\}_\mu \, |\ell_p(\Z) \|\le 1} \, \Big\|\sum_{\mu \in \Z} C_\mu \psi_\mu f\Big|B^s_{p,q}(\R)\Big\|  < \infty \, .
\eeqq
By chosing $\{C_\mu\}_\mu$ appropriate it is immediate that 
$f \in B^s_{p,q}(\R)_\unif$.
On the one hand it follows
\beqq
\Big\|\, f \, \sum_{\mu\in \Z} C_{\mu} \psi_{\mu} 
\, \Big|L_p(\R)\Big\| \lesssim \Big\|f \, \sum_{\mu\in \Z} C_{\mu} \psi_{\mu} \Big|B^s_{p,q}(\R)\Big\| \le C_f < \infty
\eeqq
for all sequences $\{C_\mu\}_\mu $, $\|\, \{C_\mu\}_\mu\, |\ell_p (\Z)\| \le 1$.
On the other hand  \eqref{k-06} yields
\beqq
\Bigg\{\sum_{k=0}^{\infty}\bigg(2^{ksp} 
\sup_{|h| < 2^{-k}}  \sum_{\mu\in \Z} && \hspace{-0.7cm} |C_{\mu}|^p    
\big\|\Delta_h^m(\psi_{\mu}f)(\cdot) \big|  L_p(\R)\big\|^p \bigg)^{q/p} \Bigg\}^{1/q} 
\\
& \lesssim & \sum_{\ell=1,\ldots,n} \Big\|  \sum_{\mu \in\Omega_\ell} C_{\mu }\psi_{\mu }f \, \Big|B^s_{p,q}(\R)\Big\|
\lesssim  C_f <\infty
\eeqq 
for all sequences $\{C_\mu\}_\mu $, $\|\, \{C_\mu\}_\mu\, |\ell_p (\Z)\| \le 1$.
In view of the definition of $M^s_{p,q}(\R)$, combined with Theorem \ref{final}, we conclude that $f\in M(B^s_{p,q}(\R))$.
\hspace*{\fill} \rule{3mm}{3mm}

%&&&&&&&&&&&&&&&&&&&&&&&&&&&&&&&&&&&&&&&&&&&&&&&&&&&&&&&&&&&&&&&&&&&&&&&&&&&&&&&&&&&
%&&&&&&&&&&&&&&&&&&&&&&&&&&&&&&&&&&&&&&&&&&&&&&&&&&&&&&&&&&&&&&&&&&&&&&&&&&&&&&&&&&&

\subsection{Proof of Theorem \ref{final2} and  Theorem \ref{final3}} \label{main5}

%&&&&&&&&&&&&&&&&&&&&&&&&&&&&&&&&&&&&&&&&&&&&&&&&&&&&&&&&&&&&&&&&&&&&&&&&&&&&&&&&&&&
%&&&&&&&&&&&&&&&&&&&&&&&&&&&&&&&&&&&&&&&&&&&&&&&&&&&&&&&&&&&&&&&&&&&&&&&&&&&&&&&&&&&

\noindent
{\bf Proof of Theorem \ref{final2}}. Because of $s>0$ the space $B^s_{\infty,q}(\R)$ forms an algebra with respect 
to pointwise multiplication, see Theorem \ref{algebra}, i.e., $B^s_{\infty,q}(\R) \hookrightarrow M( B^s_{\infty,q}(\R))$.
On the other hand, the function $g \equiv 1$ belongs to $B^s_{\infty,q}(\R)$. 
This implies that a pointwise multiplier $f$ has to belong to $B^s_{\infty,q}(\R)$ as well.
\hspace*{\fill} \rule{3mm}{3mm} 
\\

\noindent
{\bf Proof of Theorem \ref{final3}}.  
Part (i) follows from Proposition \ref{limes} and Lemma \ref{mult}.
Now we turn to (ii).
Let $s=d/p$ and $q\leq \min(1,p)$.
The needed modifications of Step 1 of the proof of Theorem \ref{final}
are based on the estimate
\beqq
\sum_{\mu\in \Z}\big\| \phi_{\mu}f \big|L_\infty(\R)\big\|^p 
\lesssim \sum_{\mu\in \Z}\big\| \phi_{\mu}f \big|B^{s}_{p,q}(\R)\big\|^p  \lesssim \big\|f\big|B^{s}_{p,q}(\R)\big\|^p,
\eeqq
where we first employed Lemma \ref{emb1} and afterwards Proposition \ref{s<p,q<u}.
This yields sufficiency. Necessity is a consequence of Lemma \ref{nec}.
\\
Finally, the arguments from the proof of Theorem \ref{nec1} carry over to this limiting situation described in (iii).
\hspace*{\fill} \rule{3mm}{3mm}

%&&&&&&&&&&&&&&&&&&&&&&&&&&&&&&&&&&&&&&&&&&&&&&&&&&&&&&&&&&&&&&&&&&&&&&&&&&&&&&&&&&&
%&&&&&&&&&&&&&&&&&&&&&&&&&&&&&&&&&&&&&&&&&&&&&&&&&&&&&&&&&&&&&&&&&&&&&&&&&&&&&&&&&&&

\subsection{Localized Besov spaces as subspaces of $M(B^s_{p,q}(\R))$ and consequences} \label{main6}

%&&&&&&&&&&&&&&&&&&&&&&&&&&&&&&&&&&&&&&&&&&&&&&&&&&&&&&&&&&&&&&&&&&&&&&&&&&&&&&&&&&&
%&&&&&&&&&&&&&&&&&&&&&&&&&&&&&&&&&&&&&&&&&&&&&&&&&&&&&&&&&&&&&&&&&&&&&&&&&&&&&&&&&&&

\begin{proposition}\label{loc}
Let $0<q<p\leq\infty$ and $0<v\leq \infty$. Let either $s >d/p$ or $s =d/p$ and $q\leq 1$. 
Then we have $$B^s_{p,q,v}(\R)\hookrightarrow M(B^s_{p,q}(\R))$$
if and only if $1/v\geq 1/q-1/p$.
\end{proposition}

\begin{proof}
{\it Step 1.} Sufficiency. Proposition \ref{s<p,q<u} yields that it is enough to  consider the case $1/v=1/q-1/p$. 
Let $s<m\leq s+1$. We shall prove that
\beqq
\|\,  fg\, |B^s_{p,q}(\R)\|_{2m} \leq C\,  \|\, g\, |B^s_{p,q}(\R)\| \, \|\,  f\, |B^s_{p,q,v}(\R)\| 
\eeqq
holds for all $g\in B^s_{p,q}(\R)$ and $f\in B^s_{p,q,v}(\R)$. 
From now on we shall follow the proof of Proposition \ref{wichtig}. Observe $B^s_{p,q,v}(\R) \hookrightarrow B^s_{p,q,\infty}(\R)=
B^s_{p,q}(\R)_\unif$.
Since the condition $p\leq q$ is not needed in  Step 1 and Substep 2.1 of the proof of Proposition \ref{wichtig}, 
it is enough to deal with the case $ m < u \le 2m$. From \eqref{ws-23}, $q< p$ and Proposition \ref{diff} we derive that
\beqq 
S_u &\leq & \Bigg\{\sum_{k=0}^{\infty}\bigg(2^{ksp} \sup_{|h| < 2^{-k}}  \sum_{\mu\in \Z}
\big\| \phi_{\mu}g \big|L_\infty(\R) \big\|^p \cdot  \big\|\Delta_h^u(\psi_{\mu}f) \big|  L_p(\R)\big\|^p \bigg)^{q/p} \Bigg\}^{1/q}
\\
&\leq & \Bigg\{\sum_{\mu\in \Z}\big\| \phi_{\mu}g \big|L_\infty(\R) \big\|^q\sum_{k=0}^{\infty} 2^{ksq} \sup_{|h| < 2^{-k}}   
\big\|\Delta_h^u(\psi_{\mu}f)\big|  L_p(\R)\big\|^q \Bigg\}^{1/q} 
\\
&\leq & \Bigg\{\sum_{\mu\in \Z}\big\| \phi_{\mu}g \big|L_\infty(\R) \big\|^q \cdot \big\| \, \psi_{\mu}f\, \big| B^s_{p,q}(\R)\big\|^q \Bigg\}^{1/q}.
\eeqq
Now H\"older's inequality with $\frac{1}{v}+\frac{1}{p}=\frac 1q$,  Lemma \ref{emb1} and Corollary \ref{klar} yield
\beqq
S_u  
&\leq & 
\Bigg\{\sum_{\mu\in \Z}\big\| \phi_{\mu}g \big|L_\infty(\R) \big\|^p\Bigg\}^{1/p} 
\Bigg\{ \sum_{\mu\in \Z} \big\| \, \psi_{\mu}f \, \big| B^s_{p,q}(\R)\big\|^v\Bigg\}^{1/v}
\\
&\lesssim &  \bigg(\sum_{\mu\in \Z}\big\| \phi_{\mu}g \big|B^s_{p,p}(\R) \big\|^p\bigg)^{1/p} \,  \big\| \, f\,\big| B^s_{p,q,v}(\R)\big\|\\
&\lesssim &  \|\, g\, |B^s_{p,q}(\R)\| \,  \| \, f\, |B^s_{p,q,v}(\R)\|.
\eeqq
\noindent
{\it Step 2.} Necessity.
We shall employ the same type of arguments as in Substep 2.2 of the proof of Proposition \ref{s<p,q<u}, see in particular \eqref{4.22} and \eqref{ws-13}.
Then we choose
\[
f_{M,N,\alpha}(x) := \sum_{j=M}^{N} \, \alpha_j\,  \tilde{\psi}_{1,j,\mu_j} (x)\, , \qquad \mu_j := (4^j, 0, \ldots  , 0) 
\]
for some sequence $\alpha:= (\alpha_j)_j$, $N$ and $M$ (to be fixed later).
Define $\gamma_j := 2^{j\big(s + d(\frac 1 2 - \frac 1 p)\big)} |\alpha_j|$, $j \ge M$.
Then, by making use of the same conventions as in \eqref{ws-15a}, we obtain
\[
\| \, f_{M,N,\alpha} \, | B^s_{p,q} (\R) \| \asymp  \bigg( \sum_{j=M}^N 2^{j\big(s + d(\frac 1 2 - \frac 1 p)\big)q} |\alpha_j|^q \bigg)^{1/q} =
\bigg( \sum_{j=M}^N  |\gamma_j|^q \bigg)^{1/q} 
\]
and
\[
\| \, f_{M,N,\alpha} \, | B^s_{p,q,v} (\R) \| \asymp  \bigg( \sum_{j=M}^N 2^{j\big(s + d(\frac 1 2 - \frac 1 p)\big)v} |\alpha_j|^v \bigg)^{1/v}
= \bigg( \sum_{j=M}^N  |\gamma_j|^v  \bigg)^{1/v} \, . 
\]
Defining 
\[
g_{M,N}(x) := \sum_{j =M}^N \,   \, \psi (x-2^{-j}\mu_j), \qquad x \in \R\, ,
\]
 we conclude
\[
\| \, g_{M,N} \, | B^s_{p,q} (\R) \| \asymp  \,  (N-M)^{1/p} \, .
\]
The lower bound is trivial (it even holds for $\| \, g_{M,N} \, | L_{p} (\R) \|$). For the proof of the  upper bound
one uses the information on the supports of the functions $\psi (\, \cdot \, -2^{-j}\mu_j)$ and Proposition \ref{diff}.
Observe, all hidden constants are independent of $M,N$ and $\alpha$. 
By construction we have the identity $f_{M,N,\alpha} \cdot g_{M,N} = f_{M,N,\alpha}$.
Hence, $B^s_{p,q,v} (\R) \subset M(B^s_{p,q}(\R))$ implies 
\\
$B^s_{p,q,v} (\R) \hookrightarrow M(B^s_{p,q}(\R))$ and therefore
\beqq
\| \, f_{M,N,\alpha} \, | B^s_{p,q}(\R) \|  & = & \| \, f_{M,N,\alpha} \, \cdot \, g_{M,N} \, | B^s_{p,q}(\R) \| 
\\
&\le & \, c \, \| \,f_{M,N,\alpha} \, | B^s_{p,q,v} (\R)\|
\, \, \| \, g_{M,N} \, | B^s_{p,q}(\R)\|
\eeqq
for some $c>0$. Consequently
\[
\bigg( \sum_{j=M}^N  \, |\gamma_j|^q \bigg)^{1/q} \le c' \, 
\bigg( \sum_{j=M}^N  \, |\gamma_j|^v \bigg)^{1/v} \, 
 \big(N-M\big)^{1/p} 
\]
holds for some $c'$ independent of $N$, $M$ and $\alpha$.
Choosing  $\gamma_j =1 $ for all $j$ the necessity of $\frac 1 v \ge \frac 1 q - \frac 1 p $ follows.
\end{proof}

\begin{remark} \rm 
 Proposition \ref{loc} in  case $1\leq q<p \leq \infty$ can be found in \cite{SS}.
\end{remark}

 Proposition \ref{loc} has an interesting consequence. From Lemma \ref{mult} we know that always
 $M(B^s_{p,q}(\R)) \hookrightarrow B^s_{p,q}(\R)_\unif$ holds.
Now we ask for coincidence.

\begin{corollary}\label{abc}
Let $1\le p,q \le \infty$  and $s\in \re$. 
Then
$M(B^s_{p,q}(\R))$ coincides with $B^s_{p,q} (\R)_\unif$ if and only if either $1\leq p\leq q\leq \infty$ and $s>d/p$ or $p=q=1$ and $s=d$. 
\end{corollary}

\begin{proof}
{\em Step 1.} Necessity.
 From $M(B^s_{p,q}(\R))= B^s_{p,q} (\R)_\unif$ we conclude $B^s_{p,q}(\R) \hookrightarrow M(B^s_{p,q} (\R))$. Now Theorem \ref{algebra}
 yields either $s>d/p$ or $0< p<\infty$, $0 < q \le 1$ and $s=d/p$.
 Next we employ Proposition \ref{loc}. Hence, if $s>d/p$ then  $B^{s}_{p,q} (\R)_\unif \not\subset M(B^s_{p,q}(\R))$ if $0 < q < p \le \infty$.
 Now we turn to $s=d/p$. Applying Proposition \ref{loc} once again, we find 
 $B^{d/p}_{p,1} (\R)_\unif \not\subset M(B^{d/p}_{p,1}(\R))$ if $1 < p \le \infty$.
 \\
 {\em Step 2.} Sufficiency. This follows immediately from Corollary \ref{klar} and Theorem \ref{b-main}. 
\end{proof}

\noindent
{\bf Proof of Remark \ref{abcd}}.
Theorem \ref{algebra} yields  
\[
 \|f \, \cdot \, g\, |B^s_{p,q} (\Omega) \| \le c \, \|\, f \,  |B^s_{p,q} (\Omega) \|\, \| \, g\, |B^s_{p,q} (\Omega) \|
\]
whenever this inequality is true on $\R$.
Hence, $B^s_{p,q} (\Omega) \hookrightarrow M(B^s_{p,q} (\Omega))$ under the restrictions of Theorem \ref{algebra}.
On the other hand, the function $f\equiv 1$ belongs to all spaces $B^s_{p,q} (\Omega)$, since $\Omega$ is bounded.
Hence, $f \in B^s_{p,q} (\Omega)$ is also a necessary condition for $f$ to belong to $M(B^s_{p,q} (\Omega))$.
\hspace*{\fill} \rule{3mm}{3mm} 

%&&&&&&&&&&&&&&&&&&&&&&&&&&&&&&&&&&&&&&&&&&&&&&&&&&
%&&&&&&&&&&&&&&&&&&&&&&&&&&&&&&&&&&&&&&&&&&&&&&&&&&

\end{document}